\documentclass[11pt, reqno]{article}
\usepackage{amsmath,amssymb,amsfonts,amsthm}
\usepackage{color}

\usepackage[colorlinks=true,linkcolor=blue,citecolor=blue,urlcolor=blue,pdfborder={0 0 0}]{hyperref}
\usepackage{cleveref}

\usepackage{caption}
\usepackage{thmtools}

\usepackage{mathrsfs}

\crefname{theorem}{Theorem}{Theorems}
\crefname{thm}{Theorem}{Theorems}
\crefname{lemma}{Lemma}{Lemmas}
\crefname{lem}{Lemma}{Lemmas}
\crefname{remark}{Remark}{Remarks}
\crefname{prop}{Proposition}{Propositions}
\crefname{defn}{Definition}{Definitions}
\crefname{claim}{Claim}{Claims}
\crefname{corollary}{Corollary}{Corollaries}
\crefname{conjecture}{Conjecture}{Conjectures}
\crefname{question}{Question}{Questions}
\crefname{chapter}{Chapter}{Chapters}
\crefname{section}{Section}{Sections}
\crefname{figure}{Figure}{Figures}

\theoremstyle{plain}
\newtheorem{thm}{Theorem}[section]

\newtheorem{lemma}[thm]{Lemma}

\newtheorem{corollary}[thm]{Corollary}

\newtheorem{prop}[thm]{Proposition}
\newtheorem{conjecture}[thm]{Conjecture}

\newtheorem{question}[thm]{Question}
\theoremstyle{definition}

\theoremstyle{remark}
\newtheorem{remark}[thm]{Remark}

\numberwithin{equation}{section}

\renewcommand{\P}{\mathbb P}
\newcommand{\E}{\mathbb E}

\newcommand{\R}{\mathbb R}
\newcommand{\Z}{\mathbb Z}
\newcommand{\N}{\mathbb N}

\newcommand{\cG}{\mathcal G}

\newcommand{\sA}{\mathscr A}
\newcommand{\sB}{\mathscr B}

\newcommand{\sD}{\mathscr D}
\newcommand{\sE}{\mathscr E}

\newcommand{\sG}{\mathscr G}

\newcommand{\bbH}{\mathbb H}

\usepackage[margin=1.025in]{geometry}
\usepackage{extarrows}
\usepackage{commath}
\usepackage{mathtools}
\newcommand{\eps}{\varepsilon}
\usepackage{bbm}
\usepackage{setspace}
\usepackage{enumitem}
\usepackage{tikz}

\usepackage{cite}

\newcommand{\Aut}{\operatorname{Aut}}

\newcommand{\bP}{\mathbf P}
\newcommand{\bE}{\mathbf E}

\newcommand{\opleq}{\preccurlyeq}
\newcommand{\opgeq}{\succcurlyeq}
\newcommand{\myfrac}[3][0pt]{\genfrac{}{}{}{}{\raisebox{#1}{$#2$}}{\raisebox{-#1}{$#3$}}}

\newcommand{\Aint}{A^\mathrm{int}}
\newcommand{\Cint}{C^\mathrm{int}}
\newcommand{\Sint}{S^\mathrm{int}}
\newcommand{\Bint}{B^\mathrm{int}}

\newcommand{\rad}{\operatorname{rad}}
\newcommand{\diam}{\operatorname{diam}}

\title{The $L^2$ boundedness condition in nonamenable percolation}

\author{Tom Hutchcroft\footnote{Statistical Laboratory, DPMMS, University of Cambridge. Email: \href{mailto:t.hutchcroft@maths.cam.ac.uk}{t.hutchcroft@maths.cam.ac.uk}}}

\begin{document}

\maketitle

\setstretch{1.2}

\begin{abstract}
Let $G=(V,E)$ be a connected, locally finite, transitive graph, and consider Bernoulli bond percolation on $G$. In recent work, we conjectured that if $G$ is nonamenable then the matrix of critical connection probabilities 
$T_{p_c}(u,v)=\P_{p_c}(u\leftrightarrow v)$ 
is bounded as an operator $T_{p_c}:L^2(V)\to L^2(V)$ and proved that this conjecture holds for several classes of graphs, including all transitive, nonamenable, Gromov hyperbolic graphs.
In notation, the conjecture states that $p_c<p_{2\to 2}$, where for each $q\in [1,\infty]$ we define $p_{q\to q}$ to be the supremal value of $p$ for which the operator norm $\|T_p\|_{q\to q}$ is finite. 
  We also noted in that work that the conjecture implies two older conjectures, namely that percolation on transitive nonamenable graphs always has a nontrivial nonuniqueness phase, and that critical percolation on the same class of graphs has mean-field critical behaviour.

 In this paper we further investigate the consequences of the $L^2$ boundedness conjecture. In particular, we prove that the following hold for all transitive graphs:
   i) The two-point function decays exponentially in the distance for all $p<p_{2\to 2}$; ii) If $p_c<p_{2\to 2}$, then the critical exponent governing the extrinsic diameter of a critical cluster is $1$; iii) Below $p_{2\to 2}$, percolation is ``ballistic" in the sense that the intrinsic (a.k.a.\ chemical) distance between two points is exponentially unlikely to be much larger than their extrinsic distance; 
   iv) If $p_c<p_{2\to 2}$, then $\|T_{p_c}\|_{q\to q} \asymp (q-1)^{-1}$ and $p_{q\to q}-p_c \asymp q-1$ as $q\downarrow 1$; v) If $p_c<p_{2\to 2}$, then various `multiple-arm' events have probabilities comparable to the upper bound given by the BK inequality. In particular, the probability that the origin is a trifurcation point is of order $(p-p_c)^3$ as $p \downarrow p_c$.
All of these results are new even in the Gromov hyperbolic case.

   Finally, we apply these results together with duality arguments to 
   compute the critical exponents governing the geometry of intrinsic geodesics at the \emph{uniqueness threshold} of percolation in the hyperbolic plane.

\end{abstract}

\newpage

\section{Introduction}

Let $G=(V,E)$ be a connected, locally finite graph and let $p\in [0,1]$. In \textbf{Bernoulli bond percolation}, we independently declare each edge of $G$ to either be open or closed, with probability $p$ of being open, and let $G[p]$ be the subgraph of $G$ formed by deleting every closed edge and retaining every open edge. 
The connected components of $G[p]$ are referred to as \textbf{clusters}. 
We are particularly interested in \emph{phase transitions}, which occur when qualitative features of $G[p]$ change abruptly as $p$ is varied through some critical value. 
The two most interesting such phase transitions occur at the \textbf{critical probability}, which is defined to be 
\[
p_c=p_c(G)=\inf\bigl\{p\in [0,1]: G[p] \text{ has an infinite cluster a.s.}\bigr\},
\]
and the \textbf{uniqueness threshold}, which is defined to be
\[
p_u=p_u(G)=\inf\bigl\{p\in [0,1]: G[p] \text{ has a unique infinite cluster a.s.}\bigr\}.
\]
Many of the most interesting questions in percolation theory concern the behaviour of percolation at and near these critical values.

Historically, percolation was studied primarily on Euclidean lattices such as the hypercubic lattice $\Z^d$. In this case, Aizenman, Kesten, and Newman \cite{MR901151} proved that there is at most one infinite cluster and hence that $p_c=p_u$. An alternative proof of the same fact was later found by Burton and Keane \cite{burton1989density}.  
In their celebrated paper \cite{bperc96}, Benjamini and Schramm proposed a systematic study of percolation on general \textbf{quasi-transitive} graphs, that is, graphs for which the action of the automorphism group on the vertex set has at most finitely many orbits. (In other words, graphs for which `there are only finitely many different types of vertices'.) They made the following conjecture.

\begin{conjecture}
\label{conj:pcpu}
Let $G$ be a connected, locally finite, quasi-transitive graph. Then $p_c(G)<p_u(G)$ if and only if $G$ is nonamenable.
\end{conjecture}

Here, a graph $G=(V,E)$ is said to be \textbf{nonamenable} if there exists a positive constant $c$ such that $|\partial_E K| \geq c \sum_{v\in K} \deg(v)$ for every finite set of vertices $K$ in $G$, 
where $\partial_E K$ denotes the set of edges with one endpoint in $K$ and one in $V\setminus K$. 
Both proofs of uniqueness in $\Z^d$ can be generalized to show that $p_c(G)=p_u(G)$ for every amenable quasi-transitive graph, so that only the `if' direction of \cref{conj:pcpu} is open. 
While various special cases of \cref{conj:pcpu} have now been verified \cite{MR1756965,Hutchcroftnonunimodularperc,1804.10191,MR1833805,BS00,MR1614583,MR2221157,MR3009109}, it remains completely open in general.
We refer the reader to \cite{grimmett2010percolation,1707.00520,heydenreich2015progress} for background on percolation, and to \cite{Haggstrom06} and the introduction of \cite{1804.10191} for a survey of progress on \cref{conj:pcpu} and related problems.


\cref{conj:pcpu} is closely related to understanding \emph{critical exponents} in percolation on nonamenable graphs. It is strongly believed that critical percolation on nonamenable transitive graphs has \emph{mean-field behaviour}. Roughly speaking, this means that critical percolation on a nonamenable transitive graph should behave similarly to critical \emph{branching random walk} on the same graph. In particular, one expects that various critical exponents exist and take the same value as on a $3$-regular tree, so that for example we expect that
\begin{align}
\bE_p|K_v| &\asymp (p_c-p)^{-1}  &p &\uparrow p_c,
\label{exponent:gamma}
\\
\bP_{p}\left(|K_v|=\infty\right) &\asymp p-p_c &p &\downarrow p_c, \text{ and}
\label{exponent:theta}
\\
\bP_{p_c}\left( |K_v| \geq n\right) & \asymp n^{-1/2} &n&\uparrow \infty,
\label{exponent:volume}
\end{align}
where $v$ is a vertex, $K_v$ is the cluster of $v$, and we write $\asymp$ to denote an equality that holds up to positive multiplicative constants in the vicinity of the relevant limit point. A well-known signifier of mean-field behaviour is the \emph{triangle condition}, which is said to hold if
\[
\nabla_{p_c}(v) := \sum_{u,w}\bP_p(v \leftrightarrow u)\bP_p(u \leftrightarrow w)\bP_p(w \leftrightarrow v) < \infty
\]
for every vertex $v$ of $G$. The triangle condition was introduced by Aizenman and Newman \cite{MR762034}, and has subsequently been shown to imply that various exponents exist and take their mean-field values \cite{MR762034,MR1127713,MR923855,MR2551766,MR2748397,Hutchcroftnonunimodularperc}. In particular, if $G$ is a connected, locally finite, quasi-transitive graph that satisfies the triangle condition then the estimates \eqref{exponent:gamma}, \eqref{exponent:theta}, and \eqref{exponent:volume} all hold for every vertex $v$ of $G$. (The lower bounds of \eqref{exponent:gamma}, \eqref{exponent:volume}, and \eqref{exponent:theta} hold on every quasi-transitive graph \cite{aizenman1987sharpness,MR762034}.) See \cite[Chapter 10]{grimmett2010percolation}  and \cite[Section 7]{Hutchcroftnonunimodularperc} for an overview of the relevant literature.
 The following conjecture is widely believed.

 \begin{conjecture}
 \label{conj:triangle}
Let $G$ be a connected, locally finite, nonamenable, quasi-transitive  graph. Then $\nabla_{p_c}(v)<\infty$ for every $v\in V$.
\end{conjecture}

As with \cref{conj:pcpu}, this conjecture is known to hold in a variety of special cases \cite{MR1756965,Hutchcroftnonunimodularperc,1804.10191,MR1833805,BS00,MR1614583,kozma2011percolation} but remains completely open in general. (Some much weaker unconditional results on the critical behaviour of percolation in the nonamenable setting have recently been obtained in \cite{hutchcroft2018locality}.)
In the Euclidean context, the triangle condition is believed to hold on $\Z^d$ if and only if $d>6$, and was proven to hold for large $d$ in the milestone work of Hara and Slade \cite{MR1043524}. The record is now held by Fitzner and van der Hofstad \cite{fitzner2015nearest}, who proved that it holds for all $d \geq 11$. It is reasonable to conjecture that the triangle condition holds on every transitive graph of at least seven dimensional volume growth; the assumption of nonamenability should be much stronger than necessary.

\medskip

In our recent work \cite{1804.10191}, we made the following strong quantitative conjecture which implies both \cref{conj:pcpu,conj:triangle}. Let $\tau_p(u,v)$ be the probability that $u$ and $v$ are connected in $G[p]$, which is known as the \textbf{two-point function}. 
 Given a non-negative matrix $M\in [0,\infty]^{V^2}$ indexed by the vertices of $G$ and $p,q\in [1,\infty]$, we define 
 $\|M\|_{p\to q}:=\sup\{ \|Mf\|_q / \|f\|_p : f \in L^p(V), f(v) \geq 0$ $\forall v\in V \}$, i.e., the norm of $M$ as an operator from $L^p(V)$ to $L^q(V)$. Let $T_p\in [0,\infty]^{V^2}$ be the \textbf{two-point matrix}  defined by $T_p(u,v)=\tau_p(u,v)$, and define the critical values
\[
p_{q\to q}=p_{q\to q}(G) = \sup\bigl\{ p\in [0,1] : \|T_p\|_{q\to q} <\infty\bigr\}.
\]
It is easily seen that $\|T_p\|_{1\to 1}=\|T_p\|_{\infty\to \infty} = \sup_{v\in V} \sum_{u\in V} \tau_p(v,u)=\sup_{v\in V} \bE_p |K_v|$, and therefore by sharpness of the phase transition \cite{aizenman1987sharpness,duminil2015new} that $p_c(G)=p_{1\to 1}(G)=p_{\infty\to\infty}(G)$ for every quasi-transitive graph $G$.

\begin{conjecture}[$L^2$ boundedness]
\label{conj:pcp22}
Let $G$ be a connected, locally finite, nonamenable, quasi-transitive graph. Then $p_c(G)<p_{2\to 2}(G)$.
\end{conjecture}

It is shown in \cite{1804.10191} that \cref{conj:pcp22} is implied by the weaker statement that $\| T_{p_c}\|_{2\to 2} <\infty$, and implies the stronger statement that $p_c(G)<p_{q\to q}(G)$ for every $q\in (1,\infty)$; a quantitative version of this fact is given in \cref{thm:qtoq_exponents}. To see that \cref{conj:pcp22} implies \cref{conj:pcpu}, simply note that if $p>p_u$ then $\inf_{u,v} T_p(u,v) \geq \inf_{u,v} \bP_p(u \to\infty)\bP_p(v \to\infty) >0$ by the Harris-FKG inequality, so that $\|T_p\|_{2\to 2}=\infty$ since $V$ is infinite. To see that \cref{conj:pcp22} implies \cref{conj:triangle}, note that, by definition, $\nabla_p(v)=T_p^3(v,v) \leq \|T_p^3\|_{2\to 2} \leq \|T_p\|_{2\to 2}^3$ for each $v\in V$ and $p\in [0,1]$.

\medskip

As a heuristic justification of \cref{conj:pcp22}, we note that if $G$ is a quasi-transitive graph then the two-point function of the critical branching random walk on $G$ (which coincides with the simple random walk Greens function) is bounded as an operator on $L^2(V)$ if and only if $G$ is nonamenable \cite[Theorem 10.3]{Woess}. Thus, \cref{conj:pcp22} is predicted by the general philosophy that, in the high-dimensional context, critical percolation should behave similarly to critical branching random walk.

\medskip




\cref{conj:pcp22} is known to hold in most of the special cases in which either \cref{conj:pcpu} or \cref{conj:triangle} had been established. Indeed, the proofs that $p_c<p_u$ holds under various \emph{perturbative  assumptions} such as small spectral radius \cite{MR1756965}, large Cheeger constant \cite{MR1833805}, or high girth \cite{MR3005730} also implicitly yield the stronger claim that $p_c<p_{2\to 2}$ under the same assumptions. In particular, it can be deduced from the work of Pak and Smirnova-Nagnibeda \cite{MR1756965} that every finitely generated nonamenable group has a Cayley graph for which $p_c<p_{2\to 2}$, which lends very strong evidence that the conjecture is true in general. As for  non-perturbative results, it is now known that \cref{conj:pcp22} holds under the additional assumption that $G$ is either Gromov hyperbolic \cite{1804.10191} or has a quasi-transitive nonunimodular subgroup of automorphisms  \cite{Hutchcroftnonunimodularperc}. 
The latter condition holds, for example, if $G=T_k \times H$ is the Cartesian product of a regular tree of degree $k\geq 3$ with a quasi-transitive graph $H$. 
 A notable exception is given by groups of cost $>1$, which are known to have $p_c<p_u$ \cite{MR2221157,MR3009109} but are not known to have $p_c<p_{2\to 2}$ or to satisfy the triangle condition at $p_c$. (As a modest first step in this direction, one could try to prove $p_c<p_{2\to2}$ for infinitely-ended transitive graphs.)

\medskip

In this paper, we further explore the consequences of \cref{conj:pcp22}. We find in particular that \cref{conj:pcp22} has various very strong consequences that are not known to follow from \cref{conj:pcpu} or \cref{conj:triangle} alone. Our results can briefly be summarized as follows:

\begin{enumerate}
\item \cref{thm:p2to2pexp}: If $p<p_{2\to 2}$, then the two-point function $\tau_p(u,v)$ is exponentially small in the graph distance $d(u,v)$. Thus, if $p_c<p_{2\to 2}$ then this property holds for both critical and slightly supercritical percolation.

\item \cref{thm:ext_radius}: If $p_c<p_{2\to 2}$, then the critical exponent governing the \emph{extrinsic} radius of clusters is $1$. That is, the probability that the cluster of the origin in $G[p_c]$ reaches distance at least $n$ in $G$ is of order $n^{-1}$. This is not known to follow from the triangle condition even under the assumption of nonamenability. 

\item \cref{thm:ballistic}: If $p<p_{2\to 2}$, then the clusters of $G[p]$ are `ballistic' in the following sense: For any two vertices $u$ and $v$ of $G$, conditioned on $u$ and $v$ being connected in $G[p]$, the intrinsic distance $d_\mathrm{int}(u,v)$ between $u$ and $v$ in $G[p]$ is exponentially unlikely to be much longer than $d(u,v)$. Moreover, for fixed $u$, the random variable $\sup \{ d_\mathrm{int}(u,v) / d(u,v) : v \leftrightarrow u\}$ has an exponential tail. 


\item \cref{thm:qtoq_exponents}: If $p_c<p_{2\to2}$, then we have the asymptotic estimates 
$\|T_{p_c}\|_{q\to q} \asymp (q-1)^{-1}$ and $p_{q\to q}-p_c \asymp q-1$ as $q\downarrow 1$. 
This can be used to deduce various strong quantitative estimates on critical and slightly supercritical percolation. 

\item \cref{thm:multibody,thm:multiarm}: If $p_c<p_{2\to 2}$ then the probabilities of various  `multiple arm' events are of the same order as the upper bound given by the BK inequality. In particular, we deduce that if $G$ is transitive then the probability that the origin is a trifurcation point is of order $(p-p_c)^3$ as $p\downarrow p_c$ (\cref{cor:furcations}).

\end{enumerate}

Further applications of the $L^2$ boundedness condition to slightly supercritical percolation are investigated in \cite{2002.02916,2002.02916b}.
Finally, we apply some of the results above together with duality arguments to study the critical behaviour of percolation \emph{at the uniqueness threshold} $p_u$ on nonamenable, quasi-transitive, simply connected planar maps (e.g. tesselations of the hyperbolic plane), which are always Gromov hyperbolic \cite{MR3658330} and hence have $p_c<p_{2\to 2}$ by the results of \cite{1804.10191}. In particular, we prove the following.
\begin{enumerate}[resume]
\item If $G$ is a nonamenable quasi-transitive simply connected planar map, then the probability that two neighbouring vertices are connected in $G[p_u]$ but have intrinsic distance at least $n$ is of order $n^{-1}$ (excluding pairs of vertices for which there are local obstructions that make the probability trivially equal to zero). 
 \end{enumerate}
We have organized the paper textbook-style into several short sections, each of which contains a complete treatment of one or more of the topics above. We will assume that the reader is familiar with  Fekete's subadditive lemma, the Harris-FKG inequality, the BK inequality, Reimer's inequality, and Russo's formula, referring them to \cite[Chapter 2]{grimmett2010percolation} otherwise.  




\section{Exponential decay of the two-point function}
\label{sec:extrinsic_decay}

In this section we study the consequences of the $p_c<p_{2\to 2}$ condition for the decay of the two-point function.
Let $G$ be a connected, locally finite graph. 
We say that $G$ has \textbf{exponential connectivity decay} at $p$ if the quantity
\[
\xi_p := -\limsup_{n\to\infty} \frac{1}{n}\log \sup\Bigl\{\tau_p(u,v) : d(u,v)\geq n\Bigr\}
\]
is positive. That is,  $\xi_p$ is maximal such that for every $\eps>0$ there exists $C_{p,\eps}<\infty$ such that
\[
\tau_{p}(u,v) \leq C_{p,\eps} \exp\Bigl[- (\xi_p -\eps)\cdot d(u,v)\Bigr]
\]
for all vertices $u$ and $v$ in $G$. Following \cite{MR1833805}, we define the \textbf{exponential connectivity decay threshold} to be
\[
p_\mathrm{exp}=p_\mathrm{exp}(G) = \sup\Bigl\{p \in [0,1] : G \text{ has exponential connectivity decay at $p$} \Bigr\}.
\]
It is clear that $p_\mathrm{exp}\leq p_u$ when $G$ is quasi-transitive, and it is known that this inequality is strict in some examples, such as the free product $\Z^2 * (\Z/2\Z)$, and saturated in others, such as trees. It is also classical that $p_c(G)\leq p_\mathrm{exp}(G)$ for every quasi-transitive graph $G$, see \cite[Chapter 6]{grimmett2010percolation}. 

The following conjecture is widely believed among experts but does not seem to have appeared in print.
 See \cite[Open Problems 3.1 and 3.2]{MR1833805} for further related problems.


\begin{conjecture}
\label{conj:pcpexp}
Let $G$ be a connected, locally finite, quasi-transitive, nonamenable graph. Then $p_c(G)<p_\mathrm{exp}(G)$.
\end{conjecture}

Perturbative proofs of $p_c<p_u$ typically also implicitly yield the stronger result $p_c<p_\mathrm{exp}$, and the previously mentioned papers \cite{MR1833805,bperc96,MR1756965,MR3005730} all yield that $p_c<p_\mathrm{exp}$ for the graphs they consider. The most notable non-perturbative result is due to Schonmann \cite{MR1888869}, who proved that $p_\mathrm{exp}=p_u$ for transitive, one-ended, nonamenable, planar graphs. Besides this, Schramm proved that for any transitive, unimodular, nonamenable graph, the critical two-point function is exponentially small in the distance for \emph{some} pairs of vertices, namely, the endpoints of a random walk have this property with high probability. See \cite{kozma2011percolation} for Schramm's proof and \cite{Hutchcroft2016944,1712.04911,1804.10191} for related results.

Our first result verifies \cref{conj:pcpexp} under the assumption that $p_c<p_{2\to 2}$. In particular, applying the results of \cite{1804.10191,Hutchcroftnonunimodularperc}, we deduce that the conjecture holds under the additional assumption that the graph in question is either Gromov hyperbolic or has a quasi-transitive nonunimodular subgroup of automorphisms.

\begin{thm}
\label{thm:p2to2pexp}
Let $G$ be a connected, locally finite graph. 
Then $p_{2\to 2}(G)\leq p_\mathrm{exp}(G)$.
\end{thm}

Let us mention again that Cayley graphs of groups of cost $>1$ are known to have $p_c<p_u$ \cite{MR2221157,MR3009109} but are not known to have $p_c<p_{2\to 2}$ or $p_c<p_\mathrm{exp}$.

\begin{remark}
The proof of \cref{thm:p2to2pexp} yields the quantitative bound
\begin{equation}
\tau_p(u,v)\leq 2\|T_p\|_{2\to 2} \exp\left[ - \frac{d(u,v)}{ e \|T_p\|_{2\to 2}} \right]
\end{equation}
for every $u,v \in V$ and $0<p<p_{2\to 2}$.
\end{remark}

The proof is an `$L^q$ version' of the usual proof that there is exponential connectivity decay when the susceptibility is finite \cite[Theorem 6.1]{grimmett2010percolation}.
Let $G$ be a connected, locally finite graph, let $p\in [0,1]$ and $n\geq 0$, and consider the symmetric matrices $C_{p,n},S_{p,n} \in [0,\infty]^{V^2}$ defined by
\[
C_{p,n}(u,v) = \tau_p(u,v) \mathbbm{1}(d(u,v)\geq n)
\qquad \text{ and } \qquad
S_{p,n}(u,v) = \tau_p(u,v) \mathbbm{1}(d(u,v) = n).
\]
($C$ is for `complement' and $S$ is for `sphere'.)
\cref{thm:p2to2pexp} will be deduced from the following proposition.


\begin{prop}
\label{prop:expdecay}
Let $G$ be a connected, locally finite graph. Then for every $q\in [1,\infty]$ and $0< p <p_{q\to q}(G)$ there exists $\eta_{p,q}>0$ such that
\[
\lim_{n\to\infty} \frac{1}{n} \log \|C_{p,n}\|_{q\to q} = \inf_{n\geq 1} \frac{1}{n} \log \|C_{p,n}\|_{q\to q} = -\eta_{p,q}.
\]
Moreover, $\eta_{p,q}$ satisfies $\eta_{p,q}\geq e^{-1}\|T_p\|_{q\to q}^{-1}$ when $0<p<p_{q\to q}(G)$.
\end{prop}

We begin with the following simple consequence of the BK inequality.
Given two matrices $M,N \in [-\infty,\infty]^{V^2}$, we write $M \opleq N$ if $M(u,v)\leq N(u,v)$ for every $u,v \in V$. Note that if $M,N \in [0,\infty]^{V^2}$ and $M\opleq N$ then $\|M\|_{q\to q}\leq \|N\|_{q\to q}$ for every $q\in [1,\infty]$.

\begin{lemma} 
\label{lem:SCsubmult}
Let $G$ be a connected, locally finite graph. Then 
\begin{equation}
\label{eq:SCBK2}
C_{p,n+m} \opleq C_{p,m}S_{p,n} 
\end{equation}
for every $p\in [0,1]$ and $n,m\geq 0$.
\end{lemma}

\begin{proof}[Proof of \cref{lem:SCsubmult}]
 Suppose that $u,v \in V$ have $d(u,v)\geq n+m$. On the event that $u$ is connected to $v$ in $G[p]$, there must exist $w$ with $d(u,w)=n$ such that the disjoint occurrence $\{u\leftrightarrow w\}\circ\{w \leftrightarrow v\}$ occurs, and such a $w$ must have $d(w,v)\geq m$ by the triangle inequality. Indeed, simply take some simple path from $u$ to $v$ in $G[p]$, and take $w$ to be the first vertex with $d(u,w)=n$ that this path visits. Summing over the possible choices of $w$ and applying the BK inequality yields the claim.
\end{proof}

\begin{proof}[Proof of \cref{prop:expdecay}]
Since $S_{p,n}\opleq C_{p,n}$ for every $n\geq 0$, 
\cref{lem:SCsubmult} implies that $\|C_{p,n}\|_{q\to q}$ satisfies the submultiplicative inequality $\|C_{p,n+m}\|_{q\to q} \leq \|C_{p,n}\|_{q\to q}\|C_{p,m}\|_{q\to q}$, and applying Fekete's Lemma we deduce that 
\[
\eta_{p,q} =- \lim_{n\to\infty}\frac{1}{n}\log \|C_{p,n}\|_{q\to q} = -\inf_{n\geq0} \frac{1}{n}\log \|C_{p,n}\|_{q\to q}
\] 
is well-defined when $0<p<p_{q\to q}$.  It remains to prove that $\eta_{p,q}>0$ for $0\leq p< p_{q\to q}$.

Let $n,m\geq 0$ and let $0<p <p_{q\to q}$.
Noting that $C_{p,\ell}\opleq C_{p,r}$ whenever $\ell \geq r$ and applying \eqref{eq:SCBK2}, we obtain that
\[
C_{p,n+m} \opleq \frac{1}{m+1}\sum_{r=0}^m C_{p,n+r} \opleq  \frac{1}{m+1} C_{p,n}\sum_{r=0}^m  S_{p,r}  \opleq \frac{1}{m+1} C_{p,n} T_p,
\]
and hence that
\[
\|C_{p,n+m}\|_{q\to q} \leq \frac{\|T_p\|_{q\to q}}{m+1} \|C_{p,n}\|_{q\to q}.
\]
%
Applying this bound inductively and using that $C_{p,0}=T_p$ we deduce that
\[
\|C_{p,kn}\|_{q\to q}\leq \left[\frac{\|T_p\|_{q\to q}}{n+1}\right]^k\|T_p\|_{q\to q}
\]
for every $k,n\geq 0$, and hence that 
\[\eta_{p,q}\geq \frac{1}{n} \log \frac{n+1}{\|T_p\|_{q\to q}}\]
for every $n\geq 1$. Taking $n=\lceil e\|T_p\|_{q\to q}\rceil -1$ completes the proof. Note that we also obtain the explicit bound
\begin{equation}
\label{eq:extdecayexplicit}
\|C_{p,r}\|_{q\to q} \leq \|T_p\|_{q\to q} \exp\left[ - \left\lfloor \frac{r}{e \|T_p\|_{q\to q}}\right\rfloor \right] \leq 2 \|T_p\|_{q\to q} \exp\left[ -  \frac{r}{e \|T_p\|_{q\to q}} \right].
\end{equation}
(We have bounded $e^{1/e}$ by $2$ for aesthetic reasons.)
\end{proof}

\begin{proof}[Proof of \cref{thm:p2to2pexp}]
Let $u,v\in V$ be such that $d(u,v)\geq n$, let $q\in [1,\infty]$, and let $0<p<p_{q\to q}$. Then we have by H\"older's inequality that
\[\tau_p(u,v) = \langle C_{p,n} \mathbbm{1}_u, \mathbbm{1}_v \rangle \leq \|C_{p,n}\mathbbm{1}_u\|_q \|\mathbbm{1}_v\|_{\frac{q}{q-1}} \leq \|C_{p,n}\|_{q\to q}.\]
It follows that $\xi_p \geq \eta_{p,q} \geq e^{-1} \|T_p\|_{q\to q}^{-1} >0$ for every $0<p<p_{q\to q}(G)$.
\end{proof}

\section{Cluster ballisticity and the extrinsic radius exponent}
\label{sec:intrinsic_decay}

We now turn to our results concerning cluster ballisticity and the extrinsic radius exponent. We begin by discussing the extrinsic radius exponent. 
Let $G=(V,E)$ be a connected, locally finite graph and let $p\in [0,1]$. Let $d$ denote the graph distance on $G$, and let $d_\mathrm{int}$ denote the graph distance on the subgraph $G[p]$, where we set $d_\mathrm{int}(u,v)=\infty$ if $u$ and $v$ are in different clusters. The \textbf{extrinsic radius} $\operatorname{rad}(K_v)$ of the cluster of $v$ in $G[p]$ is defined to be $\sup\{d(v,u) : u \in K_v\}$, while the \textbf{intrinsic radius} $\operatorname{rad}_\mathrm{int}(K_v)$ is defined to be $\sup\{d_\mathrm{int}(v,u) : u \in K_v\}$.

As discussed in the introduction, mean-field theory predicts that critical percolation on sufficiently high-dimensional quasi-transitive graphs (say, of at least seven-dimensional volume growth) resembles the trace of a critical branching random walk on the graph. Many of the specific predictions suggested by this heuristic are known to hold for any quasi-transitive graph satisfying the triangle condition. For example, Barsky and Aizenman \cite{MR1127713} showed that in this case
\[
\bP_{p_c}\bigl(|K_v|\geq n \bigr) \asymp n^{-1/2},
\] 
and Kozma and Nachmias \cite{MR2551766} showed that 
\begin{equation}
\label{eq:KNint}
\bP_{p_c}\bigl(\operatorname{rad}_\mathrm{int}(K_v)\geq n \bigr) \preceq n^{-1}.
\end{equation}
The matching lower bound \[\bP_{p_c}\bigl(\operatorname{rad}_\mathrm{int}(K_v)\geq n \bigr) \succeq n^{-1}\] holds on \emph{every} quasi-transitive graph, see \cite[Proposition 4.2]{2002.02916}. 
These exponents are the same as those governing a critical Galton-Watson tree of finite variance, which arises as the genealogical tree of a critical branching random walk.

Determining the tail of the \emph{extrinsic} radius is a more subtle matter, since it varies from graph to graph even in the mean-field regime and therefore cannot be deduced from the triangle condition alone: on a regular tree of degree at least three we have that extrinsic and intrinsic distances coincide and hence that $\bP_{p_c}\bigl(\operatorname{rad}(K_v)\geq n \bigr)=\bP_{p_c}\bigl(\operatorname{rad}_\mathrm{int}(K_v)\geq n \bigr) \asymp n^{-1}$, while for high-dimensional Euclidean lattices it was proven in the breakthrough  work of Kozma and Nachmias \cite{MR2748397} that $\bP_{p_c}\bigl(\operatorname{rad}(K_v)\geq n \bigr) \asymp n^{-2}$. Intuitively, the difference between these two exponents is explained by the fact that simple random walk is ballistic on the tree and diffusive on $\Z^d$. Since random walk is ballistic on every quasi-transitive nonamenable graph, it is reasonable to conjecture that $\bP_{p_c}\bigl(\operatorname{rad}(K_v)\geq n \bigr) \asymp n^{-1}$ for every such graph. Our next result verifies this conjecture under the assumption that $p_c<p_{2\to 2}$.

\begin{thm}
\label{thm:ext_radius} Let $G$ be a connected, locally finite, nonamenable, quasi-transitive graph, and suppose that $p_c(G)<p_{2\to 2}(G)$. Then for every $v\in V$ we have that 
\begingroup
\addtolength{\jot}{0.5em}
\begin{align}
\bP_{p_c}(\operatorname{rad}(K_v)\geq n ) 
&\asymp n^{-1}  &\text{as }n &\uparrow \infty.
\label{exponent:extradius}
\end{align}
\endgroup
\end{thm}

Note that the upper bound of \eqref{exponent:extradius} follows from the corresponding bound for the intrinsic radius, so that it remains only to prove the lower bound. \cref{thm:ext_radius} was known in the case that $G$ has a quasi-transitive nonunimodular subgroup, using arguments specific to that setting \cite[Theorem 1.6]{Hutchcroftnonunimodularperc}. We believe the theorem is new in essentially all other cases, including for tessellations of the hyperbolic plane.

The key step in the proof of \cref{thm:ext_radius} is the following intrinsic variant on \cref{prop:expdecay}, which can be used to derive various ballisticity results for percolation below $p_{2\to 2}$. For each $p\in [0,1]$ and $n\geq 0$ we define the matrix $\Cint_{p,n}\in [0,\infty]^{V^2}$ by
\[
\Cint_{p,n}(u,v) = \bP_p\left( u \leftrightarrow v, d_\mathrm{int}(u,v)\geq n\right).
\]

\begin{prop}
\label{prop:intdecay}
 Let $G$ be a connected, locally finite graph, and let $p<p_{q\to q}(G)$. Then
\[
\|\Cint_{p,n}\|_{q\to q} \leq 3 \|T_p\|_{q\to q} \exp\left[ -\frac{n}{e\|T_p\|_{q\to q}}\right]
\]
for every $n\geq 0$.
\end{prop}

Before proving \cref{prop:intdecay}, let us see how it implies \cref{thm:ext_radius}. 

\begin{proof}[Proof of \cref{thm:ext_radius} given \cref{prop:intdecay}]
Since $p_c<p_{2\to 2}$ implies that $\nabla_{p_c}<\infty$, we have that $\bP_{p_c}(\rad_\mathrm{int}(K_v) \geq n) \asymp n^{-1}$. Therefore, since the extrinsic radius is bounded by the intrinsic radius, it suffices to prove the lower bound. For every $n\geq 1$ and $m\geq n$ we trivially have that
\begin{align*}
\bP_{p_c}(\rad(K_v) \geq n) \geq \bP_{p_c}\bigl(\rad_\mathrm{int}(K_v) \geq m\bigr) - \bP_{p_c}\Bigl(\rad_\mathrm{int}(K_v) \geq m \text{ and } \rad(K_v) \leq n\Bigr).
\end{align*}
The second term on the right can be bounded by
\begin{multline*}
\bP_{p_c}\Bigl(\rad_\mathrm{int}(K_v) \geq m \text{ and } \rad(K_v) \leq n\Bigr) \leq \sum_{u\in B(v,n)} \bP_{p_c}(u \leftrightarrow v, d_\mathrm{int}(u,v)\geq m)\\ = \langle \Cint_{p,m} \mathbbm{1}_{B(v,n)}, \mathbbm{1}_v \rangle \leq \| \Cint_{p,m}\|_{2\to 2} \|\mathbbm{1}_{B(v,n)}\|_2 \|\mathbbm{1}_v\|_2,
\end{multline*}
and applying \cref{prop:intdecay} we deduce that
\[
\bP_{p_c}(\rad_\mathrm{int}(K_v) \geq m \text{ and } \rad(K_v) \leq n)
\leq 3 \|T_{p_c}\|_{2\to 2} \exp\left[ - \frac{m}{e \|T_{p_c}\|_{2\to 2}}\right] |B(v,n)|^{1/2}
\]
for every $m \geq n \geq 1$. Since $|B(v,n)|$ grows at most exponentially in $n$, it follows that there exist constants $C$ and $c$ such that if $m\geq Cn$ then
$\bP_{p_c}(\rad_\mathrm{int}(K_v) \geq m \text{ and } \rad(K_v) \leq n) \leq e^{-cn}$. 
Taking $m=\lceil C n \rceil$, we deduce that there exists a positive constant $C'$ such that
\[
\bP_{p_c}(\rad(K_v) \geq n) \geq \bP_{p_c}(\rad_\mathrm{int}(K_v) \geq C n) - e^{-cn} \geq C' n^{-1} - e^{-cn},
\]
from which the claim follows immediately. 
\end{proof}

\cref{prop:intdecay} can also be used to derive various other results concerning the ballisticity of percolation clusters. Examples of two such statements are given in the following theorem, whose proof is very straightforward and is omitted. 

\begin{thm}
\label{thm:ballistic}
Let $G$ be a connected, locally finite, quasi-transitive graph. Then for every $p<p_{2\to 2}(G)$, there exist positive constants $c_p,C_p$ and $\lambda_{p}$ such that
\begin{equation}
\bP_p\left( \max_{v\in K_u} \frac{d_\mathrm{int}(u,v)}{d(u,v)} \geq \lambda \right) \leq C_p e^{-c_p \lambda}
\end{equation}
for every $u\in V$ and $t\geq 1$, and similarly
\begin{equation}
\label{eq:exponential_distance_decay}
\bP_p\bigl( d_\mathrm{int}(u,v) \geq n \mid u \leftrightarrow v \bigr) \leq C_p e^{-c_p n }
\end{equation}
for every $u,v\in V$ and $n \geq \lambda_{p} d(u,v)$.
\end{thm}

Note that the special case of \cref{thm:ballistic} in which $p<p_c$ can also be deduced from the fact that the cluster volume in subcritical percolation has an exponential tail \cite{MR762034} (although the argument given here will give better control on the rate of divergence of the constants as $p\uparrow p_c$). See \cite{MR1404543} for a related result concerning \emph{supercritical} percolation on Euclidean lattices. 

In contrast to these results, in \cref{sec:planar} we show that for percolation in the hyperbolic plane, the quantity $\bP_{p_u}\bigl( d_\mathrm{int}(u,v) \geq n \mid u \leftrightarrow v \bigr)$ can have a power law tail at $p_u$.

\medskip

We now turn to the proof of \cref{prop:intdecay}. We write $B_\mathrm{int}(v,n)=\{ u\in V: d_\mathrm{int}(u,v)\leq n\}$ and $\partial B_\mathrm{int}(v,n)=\{ u\in V: d_\mathrm{int}(u,v)= n\}$ for the intrinsic ball and sphere of radius $n$ around $v$ respectively, and for each $p\in [0,1]$ and $m\geq n \geq 0$ define the matrices
\begin{align*}
\Bint_{p,n}(u,v) &= \P_p\left( v \in B_\mathrm{int}(u,n) \right), &&\\
\Sint_{p,n}(u,v) &=  \P_p\left(v \in \partial B_\mathrm{int}(u,n) \right), &&\text{and}\\
\Aint_{p,n,m}(u,v) &=  \P_p\left( v \in  B_\mathrm{int}(u,m) \setminus  B_\mathrm{int}(u,n-1) \right).&&
\end{align*}

\begin{lemma}
\label{lem:intrinsicsubmult}
Let $G$ be a connected, locally finite graph. Then
\begin{equation}
\label{eq:intrinsicsubmult}
\Aint_{p,n,n+m} \opleq \Bint_{p,m}\Sint_{p,n} 
\end{equation}
for every $p\in [0,1]$ and $m,n\geq 0$.
\end{lemma}

Since $\Bint_{p,n+m}=\Bint_{p,n-1}+\Aint_{p,n,n+m}$ for every $n,m \geq 0$, \cref{lem:intrinsicsubmult} implies that if $G$ is a connected, locally finite graph then the submultiplicative-type inequality
\begin{equation}
\label{eq:Bint_submult}
\Bint_{p,n+m} \opleq \Bint_{p,m}\Sint_{p,n} + \Bint_{p,n-1} \opleq \Bint_{p,m} \Bint_{p,n}
\end{equation}
holds for every $p\in [0,1]$ and $m,n\geq 0$.

\begin{remark}
It may be instructive for the reader to reflect on why Reimer's inequality does \emph{not} imply the submultiplicative inequality $\Cint_{p,n+m} \opleq \Cint_{p,n}\Cint_{p,m}$.
\end{remark}

\begin{proof}
For each three vertices $u,v,w$, let $\sG_{n,m}(u,v,w)$ be the event that $w$ is connected to $u$ with $n \leq d_\mathrm{int}(u,w)\leq n+m$ and that $v$ is the $n$th vertex on some geodesic from $u$ to $w$ in $G[p]$. The definitions ensure that
\begin{align*}\Aint_{p,n,n+m}(u,w)\leq \sum_{v\in V}\bP_p\left[\sG_{n,m}(u,v,w)\right].
\end{align*}
We now claim that 
\[\sG_{n,m}(u,v,w) \subseteq \{v \in \partial B_\mathrm{int}(u,n) \} \circ \{w \in B_\mathrm{int}(v,m) \}. \]
Indeed, on the event $\sG_{n,m}(u,v,w)$, the subgraph of $G$ induced by the intrinsic ball of radius $n$ around $u$ together with all the closed edges incident to this ball is a witness for $\{v \in \partial B_\mathrm{int}(u,n) \}$, while any geodesic in $G[p]$ from $v$ to $w$ is a witness for $\{w \in B_\mathrm{int}(v,m) \}$ that is disjoint from the former witness. Thus, we can apply Reimer's inequality to obtain that
\begin{align*}\Aint_{p,n,n+m}(u,w) \leq \sum_{v\in V} \P_p\left[\{v \in \partial B_\mathrm{int}(u,n) \} \circ \{w \in B_\mathrm{int}(v,m) \}\right] \leq \sum_{v\in V}\Bint_{p,m}(v,w)\Sint_{p,n}(u,v),
\end{align*}
which is equivalent to the claimed inequality \eqref{eq:intrinsicsubmult}.
\end{proof}

We now prove \cref{prop:intdecay}. 
 The proof is quite different to that of \cref{prop:expdecay} since submultiplicativity is not available, and relies instead on a certain averaging trick that will be used again in the proof of \cref{prop:lowerball}.

\begin{proof}[Proof of \cref{prop:intdecay}]
Applying \eqref{eq:intrinsicsubmult} and using the trivial inequality $\Bint_{p,m}\opleq T_p$, we have that
\[
\Aint_{p,kn,(k+1)n} \opleq T_p \Sint_{p,r}
\]
for every $k,n\geq 1$ and $r\leq kn$. Averaging this inequality over $(k-1)n\leq r \leq kn$ we obtain that
\[
\Aint_{p,kn,(k+1)n} \opleq \frac{1}{n+1}T_p \Aint_{p,(k-1)n,kn}
\]
for every $k,n\geq 1$. Applying this bound inductively we deduce that
\[
\|\Aint_{p,kn,(k+1)n}\|_{q\to q} \leq \left[\frac{\|T_p\|_{q\to q}}{n+1}\right]^k \|T_p\|_{q\to q},
\]
from which the claimed bound can be deduced by similar reasoning to that used in the proof of \cref{prop:expdecay}. (We bound the resulting constant $e^{1/e}/(1-1/e)$ by $3$ for aesthetic reasons.) 
\end{proof}

\subsection{Mean-field lower bounds on norms of ball and sphere operators}
\label{subsec:meanfieldlowerbounds}

In this subsection, we note that methods similar to those used to derive \cref{prop:expdecay,prop:intdecay} can also be used to prove the following simple mean-field lower bounds. These bounds complement those of \cite[Corollary 2.5]{1804.10191}, which states that if $G$ is a connected, locally finite graph then
\begin{equation} \|T_p\|_{q\to q} \geq \frac{1-p}{\|A\|_{q\to q} (p_{q\to q}-p)}
\label{eq:meanfieldlowerboundT}
\end{equation}
for every $0 \leq p < p_{q\to q}$, where $A$ is the adjacency matrix. 
 Similarly to the way that that corollary is used to prove the main results of \cite{1804.10191}, one could potentially use \cref{prop:lowerball} to establish that $p_c<p_{2\to 2}$ by showing that, for example, $\| \Sint_{p_c,n}\|_{2\to 2} \to 0$ as $n\to\infty$.

\begin{prop}
\label{prop:lowerball}
Let $G$ be an infinite, connected, bounded degree graph. Then we have that
\begin{equation}
\label{eq:extrinsiclowerball}
\|B_{p_{q\to q},n}\|_{q\to q} \geq n+1 \quad \text{ and } \quad  \|S_{p_{q\to q},n}\|_{q\to q} \geq 1
\end{equation}
for every $n\geq 0$ and $q\in [1,\infty]$, and similarly that
\begin{equation}
\label{eq:intrinsiclowerball}
\|\Bint_{p_{q\to q},n}\|_{q\to q} \geq n+1 \quad \text{ and } \quad  \|\Sint_{p_{q\to q},n}\|_{q\to q} \geq 1
\end{equation}
for every $n\geq 0$ and $q\in [1,\infty]$.
\end{prop}

The proof of \cref{prop:lowerball} is inspired by the proof of \cite[Theorem 1.1]{duminil2015new}. Note that the sphere bounds do not obviously imply the ball bounds when $q\notin \{1,\infty\}$.

\begin{proof}
We will prove \eqref{eq:intrinsiclowerball}, the proof of \eqref{eq:extrinsiclowerball} being similar. 
 Let $\Lambda \subset E$ be finite. We write $\{ u \xleftrightarrow{\Lambda} v\}$ for the event that $u$ and $v$ are connected by an open path using only edges of $\Lambda$, and for each $p\in [0,1]$ define the matrix $T_{p,\Lambda}(u,v)=\P_p(u \xleftrightarrow{\Lambda} v)$.
Note that, since $G$ has bounded degrees and $\Bint_{p,n}$ is symmetric, we have by the Riesz-Thorin theorem that 
\[
\|\Bint_{p,n}\|_{q\to q} \leq  \|\Bint_{p,n}\|_{1\to 1}  \leq  \sup_{v\in V} |B(v,n)|<\infty
\]
for every $n\geq 0$ and $q\in [1,\infty]$, and similarly that $\|T_{p,\Lambda}\|_{q\to q}\leq \|T_{p,\Lambda}\|_{1\to 1}<\infty$ for every $\Lambda \subset E$ finite.
 Applying Reimer's inequality as in the proof of \cref{lem:intrinsicsubmult}, we obtain that 
 \begin{align*}
T_{p,\Lambda}(u,v) &= \P_p\bigl(u \xleftrightarrow{\Lambda} v \text{ and $d_\mathrm{int}(u,v) \leq n-1$}\bigr) + \P_p\bigl(u \xleftrightarrow{\Lambda} v \text{ and $d_\mathrm{int}(u,v) \geq n$}\bigr) \\
 &\leq \P_p\bigl(u \xleftrightarrow{\Lambda} v \text{ and $d_\mathrm{int}(u,v) \leq n-1$}\bigr)
+  \sum_{w \in V} \P_p\Bigl( \{ u\xleftrightarrow{\Lambda} w\bigr\} \circ \bigl\{w \leftrightarrow v \text{ and } d_\mathrm{int}(w,v)=n\}  \Bigr)\\ &\leq \Bint_{p,n-1}(u,v) + \sum_{w\in V}T_{p,\Lambda}(u,w)\Sint_{p,n}(w,v),
 \end{align*}
 which is equivalent to the bound
\begin{equation}
\label{eq:TpLambda}
T_{p,\Lambda} \opleq \Bint_{p,n-1} + T_{p,\Lambda} \Sint_{p,n}.
\end{equation}
Let $K$ be a finite subset of $\N=\{0,1,\ldots\}$. Summing \eqref{eq:TpLambda} over $K$ and then taking norms, we obtain that 
\begin{equation}
\label{eq:beforearrangement}
|K|\cdot \|T_{p,\Lambda}\|_{q\to q} \leq \Bigl\|\sum_{n\in K} \Bint_{p,{n-1}}\Bigr\|_{q\to q} + \|T_{p,\Lambda}\|_{q\to q} \Bigl\|\sum_{n \in K} \Sint_{p,n}\Bigr\|_{q\to q}
\end{equation}
Suppose that $\|\sum_{n\in K} \Sint_{p,n}\|_{q\to q} < |K|$. Since every term in \eqref{eq:beforearrangement} is finite, we may rearrange to obtain that
\[
\|T_{p,\Lambda}\|_{q\to q} \leq \myfrac[0.3em]{\bigl\|\sum_{n\in K} \Bint_{p,{n-1}}\bigr\|_{q\to q}}{|K|-\bigl\|\sum_{n \in K} \Sint_{p,n}\bigr\|_{q\to q}}<\infty,
\]
and taking the limit as $\Lambda$ exhausts $E$ we obtain that $\|T_{p}\|_{q\to q}<\infty$ 
for every $p\in [0,1]$ such that $\|\sum_{n\in K} \Sint_{p,n}\|_{q\to q} < |K|$.
Since $\|T_{p_{q\to q}}\|_{q\to q}=\infty$ for every $q\in [1,\infty]$, we deduce that 
\begin{equation}\Bigl\|\sum_{n \in K} \Sint_{p_{q\to q},n}\Bigr\|_{q\to q} \geq |K| \qquad  \text{for every $q\in [1,\infty]$ and $K\subseteq \N$}.
\end{equation} 
The claimed inequalities \eqref{eq:intrinsiclowerball} follow by taking $K=\{0,\ldots,n\}$ and $K=\{n\}$ respectively.
\end{proof}






\section{Norm exponents}
\label{sec:norm_exponents}

In \cite[Proposition 2.3]{1804.10191}, we proved that if $G$ is quasi-transitive and $p_c(G)<p_{2\to 2}(G)$ then $p_c(G)<p_{q\to q}(G)$ for every $q\in (1,\infty)$. Our next result establishes a sharp quantitative\footnote{While the proof of \cite[Proposition 2.3]{1804.10191} can also be made quantitative, the best bound that we were able to prove via that approach was of the form $\|T_{p_c}\|_{q\to q} \leq q^C$ for some possibly very large constant $C$ depending on $\|T_{p_c}\|_{2\to 2}$.} version of this fact, which will yield very strong quantitative information about the critical two-point function.

\begin{samepage}
\begin{thm} 
\label{thm:qtoq_exponents}Let $G$ be a connected, locally finite, nonamenable, quasi-transitive graph, and suppose that $p_c(G)<p_{2\to 2}(G)$. Then 
\[
\|T_{p_c}\|_{q\to q} \asymp 
\begin{cases}
q & \text{as } q \uparrow \infty\\
(q-1)^{-1} & \text{as } q \downarrow 1
\end{cases}
\quad \text{ and }\quad
p_{q\to q} -p_c \asymp 
\begin{cases}
 q^{-1} & \text{as }q \uparrow \infty\\
 q-1 & \text{as }q \downarrow 1.
\end{cases}
\]
\end{thm}
\end{samepage}
The proof of \cref{thm:qtoq_exponents} is given in the following two subsections. This theorem is complemented by the lower bound
\begin{equation}
\label{eq:general_norm_lower_bound}
\|T_p\|_{q\to q} \geq \frac{1-p}{\|A\|_{q\to q} (p_{q\to q}-p)} \qquad \text{ for every $0<p<p_{q\to q}$}
\end{equation}
which is proven in \cite[Corollary 2.6]{1804.10191} and holds on every infinite, connected, locally finite graph, where $A$ is the adjacency matrix of the graph. 
%

An interesting corollary of \cref{thm:qtoq_exponents} is as follows.

\begin{corollary} 
\label{cor:logbound}
Let $G$ be a connected, locally finite, nonamenable, quasi-transitive graph, and suppose that $p_c(G)<p_{2\to 2}(G)$. Then there exists a constant $C$ such that
\[
\bE_{p_c}|K_v \cap W| \leq C \log |W|
\]
for every $v\in V$ and every $W \subseteq V$ with $|W|\geq 2$. In particular, there exists a constant $C$ such that $\bE_{p_c}|K_v \cap B(v,n)| \leq Cn$ for every $n\geq 1$.
\end{corollary}

\cref{prop:lowerball} shows that the bound $\bE_{p_c}|K_v \cap B(v,n)| \leq Cn$  is sharp up to the choice of constant.

\begin{proof}[Proof of \cref{cor:logbound}]
By H\"older's inequality, we have that
\[
\bE_{p_c}|K_v \cap W| = \langle T_{p_c} \mathbbm{1}_W, \mathbbm{1}_v\rangle \leq \|T_{p_c}\|_{q\to q} \|\mathbbm{1}_W \|_{q} \|\mathbbm{1}_v\|_{\frac{q}{q-1}} = \|T_{p_c}\|_{q\to q} |W|^{1/q}.
\]
The claim follows by taking $q = \log |W|$ and applying \cref{thm:qtoq_exponents}.
\end{proof}

A further corollary of \cref{thm:qtoq_exponents} concerns the  asymptotic density of slightly supercritical clusters. Let $G=(V,E)$ be a connected, locally finite, quasi-transitive graph.  For each $0<p\leq 1$, we define 
the \textbf{annealed upper logarithmic density} of clusters in $G[p]$ by 
\[
\delta_{\log}(p)=\limsup_{n\to\infty} \sup_{v\in V}\frac{\log \bE_p|K_v \cap B(v,n)|}{\log |B(v,n)|},
\]
where $B(v,n)$ denotes the ball of radius $n$ around $v$ in $G$. Note that Markov's inequality implies
\[
\limsup_{n\to\infty} \frac{\log |K_v \cap B(v,n)|}{\log |B(v,n)|} \leq \delta_{\log}(p)
\]
 that for every $v\in V$ almost surely. 

\begin{corollary}
\label{cor:density}
Let $G$ be a connected, locally finite, nonamenable, quasi-transitive graph, and suppose that $p_c(G)<p_{2\to 2}(G)$. Then
$\delta_{\log}(p) \asymp p-p_c$ as $p\downarrow p_c.$
\end{corollary}
%

 \cref{cor:density} should be compared with Lalley's result that in slightly supercritical percolation on tessellations of the hyperbolic plane, the a.s.\ Hausdorff dimension of the set of ideal-boundary accumulation points of an infinite percolation cluster tends to zero as $p_c$ is approached from above \cite{MR1873136}. (Indeed, it is possible to apply \cref{thm:qtoq_exponents}  to derive a quantitative version of Lalley's result showing that this dimension is  $\Theta(\eps)$ at $p_c+\eps$.) The proof of \cref{cor:density} is given after the proof of \cref{thm:qtoq_exponents}.



\medskip

We now begin to work towards the proof of \cref{thm:qtoq_exponents}.
We first lower bound the rate that $\|T_{p_{q\to q}}\|_{q'\to q'}$ blows up as $q' \uparrow q$. The resulting estimate can be thought of as a mean-field lower bound. Given a graph $G$, we define
\[
\gamma=\gamma(G)=\lim_{n\to\infty}\frac{1}{n}\sup_{v\in V}\log |B(v,n)|.
\]
The fact that this limit exists follows by Fekete's lemma, since the sequence $\sup_{v\in V} |B(v,n)|$ is submultiplicative. Moreover, if $G$ has degrees bounded by $M$ then we  have that $\gamma(G) \leq \log(M-1)<\infty$.

\begin{prop}
\label{prop:Tpqqq'q'_lower}
Let $G$ be a connected, bounded degree graph, let $q\in (2,\infty]$ and let $q' \in (2,q)$. Then
\[
\|T_{p_{q\to q}}\|_{q'\to q'} \geq  \frac{qq'}{e(q-q')\gamma},
\]
with the convention that the right hand side is equal to $q'/e\gamma$ when $q=\infty$.
\end{prop}

The proof of \cref{prop:Tpqqq'q'_lower} will apply the following elementary lemma.
\begin{lemma}
\label{lem:changingqgeneral}
Let $G=(V,E)$ be a graph, and let $M \in [0,\infty]^{V^2}$. Then
\[
\|M\|_{q_2\to q_2} \leq \|M\|_{q_1\to q_1} \left[\sup_{v\in V} \left|\{u: M(v,u) \neq 0 \}\right|\right]^{(q_2-q_1)/q_1q_2}
\]
for every $1\leq q_1 < q_2 \leq \infty$, with the convention that $(q_2-q_1)/q_1q_2 = 1/q_1$ if $q_2=\infty$.
\end{lemma}

The case $q_2=\infty$ will use the following similarly elementary fact.

\begin{lemma}
\label{lem:dumb}
Let $G=(V,E)$ be a graph, and let $M \in [0,\infty]^{V^2}$. Then
\[
\liminf_{q'\uparrow q} \|M\|_{q'\to q'} \geq \|M\|_{q\to q}
\]
for every $q \in (1,\infty]$.
\end{lemma}

\begin{proof}
We may assume that $\|M\|_{q\to q} >0$, since the claim is trivial otherwise. 
 Note that for each $f\in V^\R$, the norm $\|f\|_q$ is decreasing in $q$ for every $f\in V^\R$, and is continuous in $q\in [1,\infty]$ if $f$ is finitely supported (i.e., zero at all but finitely many vertices).
For each $a < \|M\|_{q\to q}$, there exists a finitely supported $f\in V^\R$ such that $\|M f\|_q/\|f\|_q \geq a$.
For such $f$ we have $\lim_{q'\uparrow q} \|f\|_{q'} = \|f\|_q$ and $\liminf_{q' \uparrow q} \|Mf\|_{q'}\geq \|Mf\|_q$. We deduce that $\liminf_{q'\uparrow q} \|M\|_{q'\to q'} \geq a$, and the claim follows since $a<\|M\|_{q\to q}$ was arbitrary.
\end{proof}

\begin{proof}[Proof of \cref{lem:changingqgeneral}]
We may assume that $\|M\|_{q_1\to q_1}<\infty$ and $\sup_{v\in V} \left|\{u: M(v,u) \neq 0 \}\right|<\infty$, since the claim is trivial otherwise. First suppose that $q_2\neq \infty$. 
Let $f\in L^{q_2}(V)$ be such that $f(v)\geq 0$ for every $v\in V$. By H\"older's inequality, we have that
\begin{align*}
\|Mf\|_{q_2\to q_2}^{q_2} &= \sum_{v \in V} \left[\sum_{u\in V} M(v,u)f(u) \right]^{q_2} \leq \sum_{v \in V} \left[\sum_{u\in V} M(v,u)f(u)^{q_2/q_1} \right]^{q_1} \left[ \sum_{u\in V} M(v,u)\right]^{q_2-q_1}.
\end{align*}
Letting $I_v \in L^\infty(V)$ be the function $I_v(u)=\mathbbm{1}(M(v,u) \neq 0)$, we can then bound 
\[
\sum_{u\in V} M(v,u) = \langle \mathbbm{1}_v, M I_v \rangle \leq \|M I_v\|_{q_2} \leq \|M\|_{q_2\to q_2} \left|\{u: M(v,u) \neq 0 \}\right|^{1/q_2},
\]
and putting these two bounds together we obtain that
\begin{align*}
\|Mf\|_{q_2\to q_2}^{q_2} &\leq \| M \|^{q_1}_{q_1\to q_1} \|f^{q_2/q_1}\|_{q_1}^{q_1} \| \|M\|_{q_2\to q_2}^{q_2-q_1} \left[\sup_{v\in V} \left|\{u: M(v,u) \neq 0 \}\right|\right]^{(q_2-q_1)/q_2}.
\end{align*}
Noting that $\|f^{q_2/q_1}\|_{q_1}^{q_1}=\|f\|_{q_2}^{q_2}$, taking the supremum over $f\in L^{q_2}(V)$, and dividing both sides by $\|M\|_{q_2\to q_2}^{q_2-q_1}$ yields the desired inequality. 
The case $q_2=\infty$ follows from the case $q_2<\infty$ by an application of \cref{lem:dumb}.
\end{proof}

We now turn to the proof of \cref{prop:Tpqqq'q'_lower}.

\begin{proof}[Proof of \cref{prop:Tpqqq'q'_lower}]
Applying \cref{lem:changingqgeneral} with $q_1=q'$, $q_2=q$, and $M=S_{p_{q\to q},n}$, we have that
\[
\|S_{p_{q\to q},n}\|_{q'\to q'} \geq \|S_{p_{q\to q},n}\|_{q\to q}  \left(\sup_{v\in V} \bigl|\partial B(v,n)\bigr|\right)^{-(q-q')/qq'} \geq \left(\sup_{v\in V} \bigl|\partial B(v,n)\bigr|\right)^{-(q-q')/qq'}
\]
for every $n \geq 0$, where the second inequality follows from \cref{prop:lowerball}. It follows from this inequality and \cref{prop:expdecay} that 
\[\frac{1}{e\|T_{p_{q\to q}}\|_{q'\to q'}} \leq \eta_{p_{q\to q},q'} \leq \frac{(q-q')\gamma}{q q'}, \]
and the claim follows immediately.
\end{proof}

We next give a related bound on the rate of change of $p_{q\to q}$ as a function of $q$. Recall that the $\xi_p$ denotes the exponential decay rate of the two-point function, as defined in \cref{sec:extrinsic_decay}. 

\begin{prop}
\label{prop:pqqupperbound}
Let $G$ be a connected, bounded degree graph, let $q\in (2,\infty]$ and  $q' \in [2,q)$, and suppose that $\xi_{p_{q\to q}}>0$. Then
\[
\log p_{q'\to q'} \leq \left[1- \frac{(q-q') \gamma}{qq' \xi_{p_{q\to q}}}\right] \log p_{q\to q} 
\]
with the convention that the prefactor on the right hand side is equal to $1-\gamma/(q'\xi_{p_{q\to q}})$ when $q=\infty$.
\end{prop}

Recall that it follows from \cref{thm:p2to2pexp} that $\xi_{p_{q\to q}}>0$ whenever $p_{q\to q}<p_{2\to 2}$.
We therefore immediately deduce the following corollary.

\begin{corollary}
\label{cor:pqqcontinuity}
Let $G$ be a connected, bounded degree graph. Then $p_{q\to q}(G)$ is a continuous function of $q$ on $[1,\infty]$.
\end{corollary}



\begin{proof}[Proof of \cref{prop:pqqupperbound}] By \cref{prop:expdecay} and \cref{prop:lowerball}, we have for each $q\in [1,\infty]$ and $p\in [0,1]$ that $p \geq p_{q\to q}$ if and only if
\[
\liminf_{n\to\infty} \frac{1}{n}\log\|S_{p,n}\|_{q\to q}\geq 0.
\]
Recall from \cite[Theorem 2.38]{grimmett2010percolation} that if $A$ is an increasing event, then $\log \bP_p(A) /\log p$ is a non-increasing function of $p\in (0,1)$. Thus, if $0<p_1 \leq p_2 < 1$ then we have that
\begin{multline}
\tau_{p_2}(u,v) \geq \tau_{p_1}(u,v)^{\log p_2 /\log p_1}\\ \geq \tau_{p_1}(u,v) \left[ \sup \{ \tau_{p_1}(x,y) : x,y\in V,\,  d(x,y)=d(u,v)\} \right]^{ \log (p_2/p_1)/\log p_1}.
\label{eq:changingexponential}
\end{multline}
Thus, it follows by definition of $\xi_{p_1}$ that
\[
S_{p_2,n} \opgeq  \exp\left[ \frac{\log (p_1/p_2)\xi_{p_1}}{\log p_1} n +o(n)\right] S_{p_1,n} \qquad \text{ as $n\to \infty$}.
\]
and hence that
\begin{equation}
\label{eq:pqqupper1}
\liminf_{n\to\infty} \frac{1}{n}\log\|S_{p_2,n}\|_{q\to q} \geq \frac{\log (p_1/p_2)\xi_{p_1}}{\log p_1} + \liminf_{n\to\infty} \frac{1}{n}\log\|S_{p_1,n}\|_{q\to q}.
\end{equation}
On the other hand, applying \cref{lem:changingqgeneral} as in the proof of \cref{prop:Tpqqq'q'_lower} we obtain that if $1\leq q' \leq q$ then
\begin{equation}
\label{eq:pqqupper2}
\liminf_{n\to\infty} \frac{1}{n}\log\|S_{p_2,n}\|_{q'\to q'} \geq \liminf_{n\to\infty} \frac{1}{n}\log\|S_{p_2,n}\|_{q\to q} - \frac{(q-q')\gamma}{qq'}.
\end{equation}
Combining \eqref{eq:pqqupper1} and \eqref{eq:pqqupper2}, we deduce that
\[
\liminf_{n\to\infty} \frac{1}{n}\log\|S_{p_2,n}\|_{q'\to q'} \geq \liminf_{n\to\infty} \frac{1}{n}\log\|S_{p_1,n}\|_{q\to q}
+ \frac{\log (p_1/p_2)\xi_{p_1}}{\log p_1}  - \frac{(q-q')\gamma}{qq'}
\]
for every $1 \leq q' \leq q$ and $0<p_1 \leq p_2<1$.
 Taking $p_1=p_{q\to q}$ and $p_2=p$ given by
\[
\log p = \left[1- \frac{(q-q') \gamma}{qq' \xi_{p_{q\to q}}}\right]\log p_{q\to q} 
\] we deduce from \eqref{eq:pqqupper1} and \eqref{eq:pqqupper2} that $\liminf_{n\to\infty} \frac{1}{n}\log\|S_{p,n}\|_{q'\to q'} \geq 0$, and hence that $p_{q'\to q'}\leq p$ as claimed.
\end{proof}

It remains to prove a complementary upper bound on $\|T_{p_c}\|_{q\to q}$ and lower bound on $p_{q\to q}$.

\begin{prop}
\label{prop:normupper}
Let $G$ be a connected, locally finite, quasi-transitive graph, and suppose that $p_c<p_{2\to 2}$. Then there exist positive constants $c$ and $C$ such that
\[
\|T_{p_c}\|_{q\to q} \leq C q \qquad \text{ and } \qquad p_{q\to q}-p_c \geq \frac{c}{q}
\]
for every $q \geq 2$. 
\end{prop}

\begin{proof}
The claimed lower bound on $p_{q\to q}-p_c$ follows immediately from the claimed upper bound on $\|T_{p_c}\|_{q\to q}$ together with \eqref{eq:general_norm_lower_bound}. 
Since $p_c<p_{2\to 2}$, we have that $\nabla_{p_c}<\infty$, and hence by the results of \cite{MR2551766,sapozhnikov2010upper} that there exists a constant $C$ such that
\[
\|\Bint_{p_c,n}\|_{1\to 1}=\sup_{v\in V} \E_{p_c}|B_\mathrm{int}(v,n)| \leq C (n+1)
\]
for every $n\geq 0$. Since $G$ is quasi-transitive we deduce that there exists a constant $C'$ such that
\begin{multline}
\sum_{m=0}^n \|\Sint_{p_c,n}\|_{1\to 1} =\sum_{m=0}^n \sup_{v\in V} \E_{p_c}|\partial B_\mathrm{int}(v,n)|\\ \leq \#\{\text{Orbits of $\Aut(G)$}\}\cdot\, \sup_{v\in V} \sum_{m=0}^n  \E_{p_c}|\partial B_\mathrm{int}(v,n)| \leq C'(n+1).
\label{eq:Sintpcsummedbound}
\end{multline}
Let $q\in (1,2)$ and let $\theta=\theta(q)\in (0,1)$ be such that $1/q=(1-\theta)/1+\theta/2$. Then we have by the Riesz-Thorin theorem that
\[
\|T_{p_c}\|_{q\to q} \leq \sum_{n\geq 0} \|\Sint_{p_c,n}\|_{q\to q} 
\leq \sum_{n\geq 0}\|\Sint_{p_c,n}\|_{1\to 1}^{1-\theta}\|\Sint_{p_c,n}\|_{2\to 2}^\theta.
\]
Using \cref{prop:lowerball} to bound $\|\Sint_{p_c,n}\|_{1\to 1}^{1-\theta}\leq \|\Sint_{p_c,n}\|_{1\to 1}$ and \cref{prop:intdecay} to bound $\|\Sint_{p_c,n}\|_{2\to 2}^\theta$, we have that
\[
\|T_{p_c}\|_{q\to q} \leq 3^\theta \|T_{p_c}\|_{2\to 2}^\theta \sum_{n\geq 0} \|\Sint_{p_c,n}\|_{1\to 1} \exp\left[-\frac{\theta n}{e\|T_{p_c}\|_{2\to 2}}\right].
\]
Summation by parts yields that
\begin{multline*}
\sum_{n=0}^N \|\Sint_{p_c,n}\|_{1\to 1} \exp\left[-\frac{\theta n}{e\|T_{p_c}\|_{2\to 2}}\right] = \exp\left[-\frac{\theta N}{e\|T_{p_c}\|_{2\to 2}}\right] \sum_{n=0}^N \|\Sint_{p_c,n}\|_{1\to 1} \\+ \left[1-\exp\left[-\frac{\theta }{e\|T_{p_c}\|_{2\to 2}}\right]\right]\sum_{n=0}^{N-1} \exp\left[-\frac{\theta n}{e\|T_{p_c}\|_{2\to 2}}\right] \sum_{m=0}^n \|\Sint_{p_c,n}\|_{1\to 1}
\end{multline*}
for each $0\leq N<\infty$. Applying \eqref{eq:Sintpcsummedbound} and sending $N\to\infty$ we obtain that
\begin{align*}
\|T_{p_c}\|_{q\to q} \leq 3^\theta \|T_{p_c}\|^\theta C' \left[1-\exp\left[-\frac{\theta }{e\|T_{p_c}\|_{2\to 2}}\right]\right]\sum_{n=0}^{\infty} \exp\left[-\frac{\theta n}{e\|T_{p_c}\|_{2\to 2}}\right] (n+1),
\end{align*}
and hence by calculus that there exists a positive constant $C''$ such that
\[
\limsup_{q\downarrow 1} \theta(q) \|T_{p_c}\|_{q\to q} \leq  C'' \|T_{p_c}\|_{2\to 2}.
\]
Since $\theta(q) \asymp q-1$ as $q\downarrow 1$, this implies the claim. 
\end{proof}

\begin{proof}[Proof of \cref{thm:qtoq_exponents}]
The upper bound on $\|T_{p_c}\|_{q\to q}$ and lower bound on $p_{q\to q}-p_c$ follow from \cref{prop:normupper}, the lower bound on $\|T_{p_c}\|_{q\to q}$ follows from \cref{prop:Tpqqq'q'_lower}, and the upper bound on $p_{q\to q}-p_c$ follows from \cref{thm:p2to2pexp} and \cref{prop:pqqupperbound}.
\end{proof}

\begin{proof}[Proof of \cref{cor:density}]
We begin with the upper bound. 
For each $p\in (0,p_{2\to 2})$ let $q(p)=\sup\{q\in [2,\infty] : p < p_{q\to q}\}$.
 If $p<p_{2\to 2}$ then for every $q\in [2,q(p))$ and $v\in V$ we have by H\"older's inequality that
\[
\bE_p|K_v \cap B(v,n)| = \langle T_p \mathbbm{1}_{B(v,n)},\mathbbm{1}_v\rangle \leq \|T_p\mathbbm{1}_{B(v,n)}\|_q\|\mathbbm{1}_v\|_{\frac{q}{q-1}} \leq \|T_p\|_{q\to q}|B(v,n)|^{1/q}.
\]
It follows that $\delta_{\log}(p) \leq q(p)^{-1}$ for every $0<p<p_{2\to 2}$, so that the claimed upper bound may be deduced immediately from \cref{thm:qtoq_exponents}.
For the lower bound, we apply \eqref{eq:changingexponential} to deduce that
\[
\bE_{p_2} |K_v \cap \partial B(v,n)| \geq \bE_{p_1} |K_v \cap \partial B(v,n)| \exp\left[\frac{\log(p_2/p_1)}{\log(1/p_1)}\xi_{p_1}n+o(n)\right]
\qquad \text{ as $n\uparrow \infty$}
\]
for every $v\in V$ and $0< p_1\leq p_2 \leq 1$. Using this together with \cref{prop:lowerball}, we easily deduce that if $p_1 \geq p_c$ then
\[
\delta_{\log}(p_2) \geq \delta_{\log}(p_1)+\frac{\xi_{p_1}\log(p_2/p_1)}{\gamma \log(1/p_1)},
\]
and hence that 
\[
\delta_{\log}(p) \geq \frac{\xi_{p_c}(p-p_c)}{\gamma p_c \log(1/p_c)}  + o(p-p_c) \qquad \text{as $p\downarrow p_c$.}
\]
Applying \cref{thm:p2to2pexp} gives that $\xi_{p_c}>0$, which completes the proof.
\end{proof}

\begin{question}
Let $G$ be a connected, locally finite, quasi-transitive graph such that $p_c<p_{2\to2}$. Must the matrix $T_{p_c}$ satisfy a weak-type $(1,1)$ estimate? If so, one would obtain an alternative proof of \cref{prop:normupper} using the Marcinkiewicz interpolation theorem.
\end{question}

\section{Multiple arms}
\label{sec:multiarm}

In this section, we prove that if $p_c<p_{2\to 2}$ then the probability of various `multiple arm' events are of the same order as the upper bound given by the estimates for the corresponding `one arm' events and the BK inequality. Besides their intrinsic interest, these results will also be applied in our study of percolation in the hyperbolic plane at the uniqueness threshold in \cref{sec:planar}. 
We write $\asymp_\ell$ to denote an equality holding to within multiplicative constants that may depend on the choice of $\ell$ but not on any of the other parameters in question.

\begin{thm}
\label{thm:multibody}
Let $G$ be a connected, locally finite, transitive graph with $\nabla_{p_c}<\infty$. For each $\ell \geq 2$ there exists a finite constant $K(\ell)$
such that 
\begin{align}
\bP_{p_c}(K_{v_1},\ldots,K_{v_\ell} \text{ are disjoint and $|K_{v_i}| \geq n_i$ for every $1\leq i \leq \ell$}) &\asymp_\ell \prod_{i=1}^\ell n_i^{-1/2}
\label{eq:multibody}
\end{align}
for every $n_1,\ldots,n_\ell\geq 1$ and every  $v_1,v_2,\ldots,v_\ell \in V$ such that $d(v_i,v_j) \geq K(\ell)$ for every $1 \leq i < j \leq \ell$.
\end{thm}

\begin{thm}
\label{thm:multiarm}
Let $G$ be a connected, locally finite, transitive graph with $p_c<p_{2\to 2}$. For each $\ell \geq 2$ there exists a finite constant $K(\ell)$
such that 
\begin{align}
\bP_p(v_1,\ldots,v_\ell \text{ are all in distinct infinite clusters}) &\asymp_\ell (p-p_c)^{\ell}
\label{eq:twoarm1}
\\
\bP_{p_c}(K_{v_1},\ldots,K_{v_\ell} \text{  are disjoint and $\rad_\mathrm{int}(K_{v_i}) \geq n$ for every $1\leq i \leq \ell$}) &\asymp_\ell n^{-\ell}
\label{eq:twoarm2}
\\
\label{eq:twoarm3}
\bP_{p_c}(K_{v_1},\ldots,K_{v_\ell} \text{  are disjoint and $\rad(K_{v_i}) \geq n$ for every $1\leq i \leq \ell$}) &\asymp_\ell  n^{-\ell}
\end{align}
for  every $p_c<p \leq 1$, every $n \geq 1$, and every  $v_1,v_2,\ldots,v_\ell \in V$ such that $d(v_i,v_j) \geq K(\ell)$ for every $1 \leq i < j \leq \ell$.
\end{thm}


The proof combines two techniques from the literature: the `inverse BK' method used in the proof of \cite[Theorem 3]{MR2748397}, which establishes a similar result for extrinsic radii in Euclidean lattices, and the \emph{ghost field} technique, which was introduced to percolation by Aizenman and Barsky \cite{aizenman1987sharpness} and, for our purposes, allows us to apply the inverse BK method to study the volume of the clusters rather than the radii. 
Let $G=(V,E)$ be a connected, locally finite graph, and let $G[p]$ be Bernoulli-$p$ bond percolation on $G$. A \textbf{ghost field} of intensity $h$ on $G$ is a random subset of $V$, independent of $G[p]$, such that each vertex $v$ of $G$ is included in $\cG$ independently at random with inclusion probability $1-e^{-h}$.
For each $v\in V$, $0<p<1$, and $h>0$, we define the \textbf{magnetization}
\begin{equation}
M_{p,h}(v)=\bP_{p,h}(v \leftrightarrow \cG) = \bE_p\left[1-e^{-h|K_v|}\right].
\end{equation}
Note that if $G$ is quasi-transitive then an easy FKG argument implies that there exist a positive constant $C$ such that
\begin{equation}
\label{eq:MagFKG}
\inf_{v\in V} M_{p,h}(v) \geq p^C \sup_{v\in V} M_{p,h}(v)
\end{equation}
for every $h>0$ and $0<p<1$. It is proven in \cite{MR1127713} (and follows from \cref{exponent:volume}) that if $G$ is quasi-transitive and satisfies the triangle condition then
\begin{equation}
\label{exponent:magnetization}
M_{p_c,h}(v) \asymp \sqrt{h} \qquad \text{ as $h\downarrow 0$},
\end{equation}
and in fact the lower bound of \eqref{exponent:magnetization} holds for every quasi-transitive graph \cite{aizenman1987sharpness}. (Some aspects of the proof of \cite{MR1127713} are specific to the case of $\Z^d$, see \cite[Section 7]{Hutchcroftnonunimodularperc} for an overview of the changes needed to handle arbitrary quasi-transitive graphs.)

We begin the proof of \cref{thm:multibody,thm:multiarm} with the following lemma, which is inspired by and based closely on \cite[Lemma 6.1]{MR2748397}. Given $0<p<1$ and $\mathbf{h}=(h_1,h_2,\ldots,h_\ell)\in (0,\infty)^\ell$, we write $\P_{p,\mathbf{h}}$ for the joint law of $G[p]$ and $\ell$ mutually independent ghost fields $\cG_1,\ldots,\cG_\ell$ of intensities $h_1,\ldots,h_\ell$.

\begin{lemma}
\label{lem:KN_inverseBK}
Let $G=(V,E)$ be a connected, graph with degrees bounded by $M$, let $\ell \geq 2$, and let $v_1,\ldots,v_\ell$ be vertices of $G$. Then 
\begin{multline*}\P_{p,\mathbf{h}}\left(\{v_1 \leftrightarrow \cG_1\} \circ \cdots \circ \{v_\ell \leftrightarrow \cG_\ell\}\right)
\\\geq \prod_{i=1}^\ell \left[\inf_{v\in V} M_{p,h_i}(v)\right]- \frac{4M}{p^2}\binom{\ell-1}{2} \prod_{i=1}^\ell\left[\sup_{v \in V} M_{p,h_i}(v)\right] \sup_{1 \leq i < j \leq \ell} T_p^2(v_i,v_j).
\end{multline*}
for every $0<p<1$ and $\mathbf{h}\in (0,\infty)^\ell$.
\end{lemma}

\begin{proof}[Proof of \cref{lem:KN_inverseBK}]
By inducting on $\ell$, it suffices to prove that
\begin{multline}
\label{eq:inverseBKinduction}
\P_{p,\mathbf{h}}\left(\{v_1 \leftrightarrow \cG_1\} \circ \cdots \circ \{v_\ell \leftrightarrow \cG_\ell\}\right)
\geq \P_{p,\mathbf{h}}\left(\{v_1 \leftrightarrow \cG_1\} \circ \cdots \circ \{v_{\ell-1} \leftrightarrow \cG_{\ell-1}\}\right)
\P_{p,\mathbf{h}}\left(v_\ell \leftrightarrow \cG_\ell\right)\\- \frac{4M}{p^2}(\ell-1) \prod_{i=1}^\ell\left[\sup_{v \in V} M_{p,h_i}(v)\right] \sup_{1 \leq i < j \leq \ell} T^2_p(v_i,v_j).
\end{multline}
for every $0<p<1$ and $\mathbf{h}\in (0,\infty)^\ell$.

For this proof, we will change notation and denote our percolation configuration by $\omega$ rather than $G[p]$.
Let $\omega_0,\omega_\infty$ be independent copies of Bernoulli-$p$ bond percolation on $G$, independent of the ghost fields $\cG_1,\ldots,\cG_\ell$, and let $\P_\otimes$ denote the joint law of all these random variables.
 Let $e_1,e_2,\ldots$ be an enumeration of the edge set of $G$, and for each $m\geq 0$ define
\[
\omega_m(e_i)=\begin{cases}
\omega_\infty(e_i) & \text{ if $i \leq m$} \\
\omega_0(e_i) & \text{ if $i>m$}.
\end{cases}
\]
For each event $\sA \subseteq \{0,1\}^E \times (\{0,1\}^V)^\ell$, and each $m \in \{0,1,\ldots\}\cup \{\infty\}$, let $\sA_m$ be the event that $(\omega_m,\cG_1,\ldots,\cG_\ell)$ satisfies $\sA$. 

Let $\sA=\{v_1 \leftrightarrow \cG_1\}\circ \cdots \circ \{v_{\ell-1} \leftrightarrow \cG_{\ell-1}\}$ and $\sB=\{v_\ell \leftrightarrow \cG_\ell\}$, and observe that
\[
\P_{p,\mathbf{h}}\left(\{v_1 \leftrightarrow \cG_1\} \circ \cdots \circ \{v_\ell \leftrightarrow \cG_\ell\}\right) = 
\P_\otimes\left(\sA_0 \circ \sB_0\right)
\]
and that
\[
\P_{p,\mathbf{h}}\left(\{v_1 \leftrightarrow \cG_1\} \circ \cdots \circ \{v_{\ell-1} \leftrightarrow \cG_{\ell-1}\}\right)
\P_{p,\mathbf{h}}\left(v_\ell \leftrightarrow \cG_\ell\right)
=\P_\otimes\left(\sA_0 \circ \sB_\infty\right) = \lim_{m\to\infty} \P_\otimes\left(\sA_0 \circ \sB_m\right).
\]
Thus, to prove \eqref{eq:inverseBKinduction}, it suffices to prove that
\[
\sum_{m=1}^\infty \left[\P_\otimes(\sA_0 \circ \sB_m) - \P_\otimes(\sA_0 \circ \sB_{m-1}) \right] \leq \frac{4M}{p^2}(\ell-1) \prod_{i=1}^\ell\left[\sup_{v \in V} M_{p,h_i}(v)\right] \sup_{1 \leq i < j \leq \ell} T^2_p(v_i,v_j).
\]

Observe that on the event $\sA_0 \circ \sB_m \setminus \sA_0 \circ \sB_{m-1}$ we must have that $\omega_\infty(e_m)=1$. Moreover, the events
\begin{align*}
\sA_0 \circ \sB_m \setminus \sA_0 \circ \sB_{m-1} \text{ holds, $\omega_0(e_m)=0$ and $\omega_\infty(e_m)=1$}
\intertext{and}
 \sA_0 \circ \sB_{m-1} \setminus \sA_0 \circ \sB_{m} \text{ holds, $\omega_0(e_m)=1$ and $\omega_\infty(e_m)=0$}
\end{align*}
have the same probability: Indeed, the proof of the BK inequality (see in particular the presentation in \cite[Section 2.3]{grimmett2010percolation}) establishes that there is a measure-preserving bijection between these sets.
Thus, we have that
\begin{align*}
\P_\otimes(\sA_0 \circ \sB_m) - \P_\otimes(\sA_0 \circ \sB_{m-1}) &= \P_\otimes(\sA_0 \circ \sB_m \setminus \sA_0 \circ \sB_{m-1})
- \P_\otimes(\sA_0 \circ \sB_{m-1} \setminus \sA_0 \circ \sB_{m})
\\
&\leq \P_\otimes(\sA_0 \circ \sB_m \setminus \sA_0 \circ \sB_{m-1}, \omega_0(e_m)=\omega_\infty(e_m)=1).
\end{align*}
We write $\sE_m$ for the event that $\sA_0 \circ \sB_m \setminus \sA_0 \circ \sB_{m-1}$ occurs and that $\omega_0(e_m)=\omega_\infty(e_m)=1$.

Suppose that $\sE_m$ occurs. Since $\sA_0 \circ \sB_m$ occurs there exists a collection of edge-disjoint $\omega_0$-open paths $\gamma_1,\ldots,\gamma_{\ell-1}$ such that $\gamma_i$ connects $v_i$ to $\cG_i$ for each $1\leq i \leq \ell-1$ and an $\omega_m$-open path $\gamma_\ell$ connecting $v_\ell$ to $\cG_\ell$ such that $\gamma_\ell$ does not traverse any of the edges in $\{e_{m+1},e_{m+2},\ldots\}$ that are traversed by one of the paths $\gamma_1,\ldots,\gamma_{\ell-1}$. (Indeed, such paths exist if and only if $\sA_0 \circ \sB_m$ holds.) On the other hand, since $\omega_0(e_m)=\omega_\infty(e_m)=1$ and $\sA_0 \circ \sB_{m-1}$ does \emph{not} hold, we must have that the edge $e_m$ is traversed both by $\gamma_\ell$ and $\gamma_j$ for some $1 \leq j \leq \ell-1$. Write $\sE_{m,j} \subseteq \sE_m$ for the event that $\sE_m$ holds and that the paths $\gamma_1,\ldots,\gamma_\ell$ can be chosen so that $e_m$ is traversed by both $\gamma_\ell$ and $\gamma_j$, so that $\sE_m = \bigcup_{j=1}^{\ell-1} \sE_{m,j}$.
On the event $\sE_{m,j}$, the events $\{v_\ell \leftrightarrow \{e_m^-,e_m^+\}$ in $ \omega_m \}$,  $\{v_j \leftrightarrow \{e_m^-,e_m^+\} $ in $\omega_{m-1}\}$, $\{ \cG_\ell \leftrightarrow \{e_m^-,e_m^+\}$ in $\omega_m \}$, and $\{ \cG_j \leftrightarrow \{e_m^-,e_m^+\}$ in $\omega_{m-1} \}$ all occur disjointly.  (Here, the notion of disjoint occurence refers to the big product space $(\{0,1\}^E)^2 \times (\{0,1\}^V)^\ell$ on which $(\omega_0,\omega_\infty,\cG_1,\ldots,\cG_\ell)$ is defined.) Moreover, the events $\{v_i \leftrightarrow \cG_i$ in $\omega_{m-1}\}$ for $i \notin \{j,\ell\}$ also occur disjointly from each other and from these events. Thus, we may apply the BK inequality to deduce that
\begin{multline*}
\P_\otimes(\sA_0 \circ \sB_m) - \P_\otimes(\sA_0 \circ \sB_{m-1}) 
\leq \sum_{j=1}^{\ell-1}\bP_p(v_\ell \leftrightarrow \{e_m^-,e_m^+\})\bP_p(v_j \leftrightarrow \{e_m^-,e_m^+\})\\
\cdot\bP_{p,h_\ell}(\cG \leftrightarrow  \{e_m^-,e_m^+\})\bP_{p,h_j}(\cG \leftrightarrow  \{e_m^-,e_m^+\})
\prod_{i \notin \{j,\ell\}} M_{p,h_i}(v_i),
\end{multline*}
which is at most
\begin{equation*}
4\sum_{j=1}^{\ell-1}\bP_p(v_\ell \leftrightarrow \{e_m^-,e_m^+\})\bP_p(v_j \leftrightarrow \{e_m^-,e_m^+\})
\prod_{i =1}^\ell \left[\sup_{v\in V}M_{p,h_i}(v)\right].
\end{equation*}
The claim now follows by noting that $\bP_p(v\leftrightarrow \{e_m^-,e_m^+\}) \leq \frac{1}{p}\bP_p(v\leftrightarrow e_m^-)$ for every $v\in V$ and $m\geq 1$ by the Harris-FKG inequality, and hence that
\begin{equation*}
\sum_{m=1}^\infty \bP_p(v_\ell \leftrightarrow \{e_m^-,e_m^+\})\bP_p(v_j \leftrightarrow \{e_m^-,e_m^+\})\\ \leq 
\frac{1}{p^2}\sum_{w \in V}\deg(w)T_p(v_\ell,w)T_p(w,v_j) \leq \frac{M}{p^2}T^2_p(v_\ell,v_j). \qedhere
\end{equation*}


\end{proof}

\begin{remark}
\cref{lem:KN_inverseBK}, and the proof of \cite[Lemma 6.1]{MR2748397} that inspired it, are remarkable as a rare instance where it is the convergence of the \emph{bubble diagram} rather than the triangle diagram that is indicative of mean-field type behaviour for percolation. In particular, if $0$ denotes the origin in $\Z^d$, it seems one should expect that
\[\bP_{p_c}\bigl(\text{there exist two disjoint open paths from $0$ to $\partial [-n,n]^d$}\bigr) \asymp \bP_{p_c}\bigl(0 \leftrightarrow \partial [-n,n]^d\bigr)^2\]
not just for $d>6$, but also for some $d$ slightly smaller than $6$, possibly including $d=5$.
\end{remark}

Next, we compare the probability that the events all occur disjointly to the probability that the events all hold with all clusters distinct. 

\begin{lemma}
\label{lem:diagrammatic}
Let $G=(V,E)$ be a connected, locally finite graph, let $\ell \geq 2$, and let $v_1,\ldots,v_\ell$ be vertices of $G$. Then 
\begin{multline*}\P_{p,\mathbf{h}}\left(K_{v_1},\ldots,K_{v_\ell} \text{ are disjoint and $v_i \leftrightarrow \cG_i$ for every $1\leq i \leq n$}\right)
\\\geq \P_p\left(\{v_1 \leftrightarrow \cG_1\} \circ \cdots \circ \{v_\ell \leftrightarrow \cG_\ell\}\right)- 2\binom{\ell-1}{2} \prod_{i=1}^\ell\left[\sup_{v \in V} M_{p,h_i}^\ell(v)\right] \sup_{1 \leq i < j \leq \ell} T^3_p(v_i,v_j)
\end{multline*}
for every $0<p<1$ and $\mathbf{h}\in (0,\infty)^\ell$. 
\end{lemma}

\begin{proof}[Proof of \cref{lem:diagrammatic}]
Let $\sA$ be the event that $K_{v_1},\ldots,K_{v_\ell}$ are disjoint and $v_i \leftrightarrow \cG_i$ for every $1\leq i \leq n$, and let $\sB$ be the event $\{v_1 \leftrightarrow \cG_1\} \circ \cdots \circ \{v_\ell \leftrightarrow \cG_\ell\}$. Clearly $\sA \subseteq \sB$. Suppose that $\sB \setminus \sA$ occurs. Since $\sB$ occurs, there must exist a collection of edge-disjoint open paths $\gamma_1,\ldots,\gamma_\ell$ such that $\gamma_i$ connects $v_i$ to a vertex of $\cG_i$. (Note that we may have that $v_i \in \cG_i$, in which case we may take $\gamma_i$ to be a degenerate length zero path.) On the other hand, since $\sA$ does \emph{not} occur, there must exist an open path $\gamma$ connecting two distinct vertices from the set $\{v_1,\ldots,v_\ell\}$. Suppose that this path $\gamma$ starts at the vertex $v_{j_1}$, and let $\gamma'$ be the segment of $\gamma$ between the last time it visits a vertex visited by $\gamma_{j_1}$ and the first subsequent time that it visits a vertex visited by one of the paths $\gamma_i$ for $i \neq j_1$. Call these two vertices $w_1$ and $w_2$, and let $j_2\neq j_1$ be such that $\gamma_{j_2}$ visits $w_2$. (It may be that $w_1=w_2$, in which case $\gamma'$ has length zero.)

Observe that, with this choice of $j_1$, $j_2,$ $w_1$ and $w_2$, we have that the events $\{v_{j_1}\to w_1\}$, $\{v_{j_2}\leftrightarrow w_2\}$, $\{w_1 \leftrightarrow w_2\}$, $\{w_1 \leftrightarrow \cG_{j_1}\}$, and $\{w_2 \leftrightarrow \cG_{j_2}\}$ all occur disjointly. Moreover, the events $\{v_i \leftrightarrow \cG_i\}$ for $i \notin \{j_1,j_2\}$ also occur disjointly from each other and from these events. Thus,  applying the BK inequality and summing over the possible choices of $j_1,j_2,w_1,$ and $w_2$, we obtain that
\begin{multline*}
\P_{p,\mathbf{h}}(\sB \setminus \sA)\\ \leq \sum_{w_1,w_2 \in V} \sum_{j_1=1}^\ell \sum_{j_2 \neq j_1} T_p(v_{j_1},w_1)T_p(w_1,w_2)T_p(w_2,v_{j_2}) M_{p,h_{j_1}}(w_1)M_{p,h_{j_2}}(w_2) \prod_{i \notin \{j_1,j_2\}} M_{p,h_i}(v_i)\\
\leq \ell(\ell-1)\sup_{1 \leq i < j \leq \ell} T^3_p(v_i,v_j) \prod_{i=1}^\ell \left[\sup_{v\in V} M_{p,h_i}(v) \right],
\end{multline*}
concluding the proof.
\end{proof}

We are now ready to prove \cref{thm:multibody,thm:multiarm}.

\begin{proof}[Proof of \cref{thm:multibody}]
It suffices to prove the lower bound, since the upper bound is an immediate consequence of \eqref{exponent:volume} and the BK inequality. 
Fix $\ell \geq 1$, and let $C_1$ be the constant from \eqref{eq:MagFKG}. Let $0<p<1$ and let $\mathbf{h}=(h_1,\ldots,h_\ell) \in (0,\infty)^\ell$. 
Let $G[p]$ be Bernoulli-$p$ bond percolation, let $\cG_1,\ldots,\cG_\ell$ be independent ghost fields of intensities $h_1,\ldots,h_\ell>0$, and write $\P_{p,\mathbf{h}}$ for the joint law of $G[p]$ and $\cG_1,\ldots,\cG_\ell$.
Since $G$ is quasi-transitive and $\nabla_{p_c}<\infty$, we have by the result of \cite{MR2779397} that the \emph{open triangle condition} also holds, so that for every $\eps>0$ there exists $r<\infty$ such that if $u,v$ have distance at least $r$ in $G$ then $T_{p_c}^3(u,v) \leq \eps$. In particular, it follows that there exists $K(\ell)$ such that if $u,v$ are vertices of $G$ with distance at least $K(\ell)$ then
\[
T_{p_c}^2(u,v) \leq T_{p_c}^3(u,v) \leq \frac{p_c^{2+C_1\ell}}{64 M \ell^2},
\]
where $M$ is the maximum degree of $G$. 
(Note that if $p_c<p_{2\to 2}$ we can take $K(\ell)$ to be $O(\ell)$, and if $G$ is transitive with $p_c<p_{2\to 2}$ we can take $K(\ell)$ to be $O(\log \ell)$.) With this choice of $K(\ell)$, we deduce from \cref{lem:KN_inverseBK,lem:diagrammatic} that if $v_1,\ldots,v_\ell$ are such that $d(v_i,v_j) \geq K(\ell)$ for every $1 \leq i < j \leq \ell$, then
\begin{equation*}\P_{p_c,\mathbf{h}}\left(K_{v_1},\ldots,K_{v_\ell} \text{ are disjoint and $v_i \leftrightarrow \cG_i$ for every $1\leq i \leq \ell$}\right)
\geq \frac{1}{2}\prod_{i=1}^\ell \left[\inf_{v\in V} M_{p,h_i}(v)\right]
\end{equation*}
for every $h_1,\ldots,h_\ell>0$. Since $G$ is quasi-transitive, we may apply the lower bound of \eqref{exponent:magnetization} to deduce that there exists a constant $c_1$ such that
\begin{equation}\P_{p_c,\mathbf{h}}\left(K_{v_1},\ldots,K_{v_\ell} \text{ are disjoint and $v_i \leftrightarrow \cG_i$ for every $1\leq i \leq \ell$}\right)
\geq c_1^\ell \prod_{i=1}^\ell \sqrt{h_i}.
\label{eq:magmultlower}
\end{equation}
On the other hand, the BK inequality implies that 
\begin{equation*}\P_{p_c,\mathbf{h}}\left(K_{v_1},\ldots,K_{v_\ell} \text{ are disjoint and $v_i \leftrightarrow \cG_i$ for every $1\leq i \leq \ell$}\right)
\leq  \prod_{i=1}^\ell \sup_{v\in V} M_{p_c,h_i}(v),
\end{equation*}
and hence by \eqref{exponent:magnetization} that, since $\nabla_{p_c}<\infty$, there exists a constant $C_2$ such that
\begin{equation}\P_{p_c,\mathbf{h}}\left(K_{v_1},\ldots,K_{v_\ell} \text{ are disjoint and $v_i \leftrightarrow \cG_i$ for every $1\leq i \leq \ell$}\right)
\leq C_2^\ell \prod_{i=1}^\ell \sqrt{h_i}
\label{eq:magmultupper}
\end{equation}
for every $h_1,\ldots,h_\ell >0$.

It remains to convert the magnetization estimates \eqref{eq:magmultlower} and \eqref{eq:magmultupper} into the claimed estimate \eqref{eq:multibody}. Let $\delta_\ell>0$ be a sufficiently small that
\[
\frac{1-e^{-\delta_\ell}}{(1-e^{-1})\sqrt{\delta_\ell}} \leq \frac{c_1^\ell}{2 \ell C_2^\ell},
\]
where $c_1$ and $C_2$ are the constants from \eqref{eq:magmultlower} and \eqref{eq:magmultupper} respectively.
Let $n_1,n_2,\ldots,n_\ell \geq 1$ and let $h_i=\delta_\ell/n_i$ for each $1\leq i \leq \ell$. Let $\sD$ be the event that $K_{v_1},\ldots,K_{v_\ell}$ are disjoint, let $\sA$ be the event that $v_i \leftrightarrow \cG_i$ for every $i \geq 1$, and let $\sB$ be the event that $|K_{v_i}|\geq n_i$ for every $i \geq 1$. Then we have that
\begin{equation*}
\P_{p_c,\mathbf{h}}(\sA \cap \sD)
\leq \P_{p_c,\mathbf{h}}\left(\sB \cap \sD \right) + \sum_{j=1}^\ell \P(\sA \cap \sD \cap \{|K_{v_j}|\leq n_j\}).
\end{equation*}
For each $1 \leq i,j \leq \ell$ let $h_{i,j}$ be defined by $h_{i,j}=h_i$ if $i \neq j$ and $h_{j,j}=1/n_j$, and let $\mathbf{h}_j=(h_{1,j},\ldots,h_{\ell,j})$. Then for each $1\leq j \leq n$ we can write
\begin{align*}
\P_{p_c,\mathbf{h}}(\sA \cap \sD \cap \{|K_{v_j}|\leq n_j\}) &= \bE_{p_c}\left[\mathbbm{1}\left(\sD \cap \{|K_{v_j}|\leq n_j\}\right)\prod_{i=1}^\ell\left(1-e^{-h_i |K_{v_i}|}\right)\right]
\\
&= \bE_{p_c}\left[\mathbbm{1}\left(\sD \cap \{|K_{v_j}|\leq n_j\}\right)\frac{1-e^{-h_{j} |K_{v_j}|}}{1-e^{-h_{j,j} |K_{v_j}|}}\prod_{i=1}^\ell\left(1-e^{-h_{i,j} |K_{v_i}|}\right)\right]
\end{align*}
from which it follows that
\begin{align*}
\P_{p_c,\mathbf{h}}(\sA \cap \sD \cap \{|K_{v_j}|\leq n_j\})
&\leq \frac{1-e^{-\delta_\ell}}{1-e^{-1}}\bE_{p_c}\left[\mathbbm{1}\left(\sD\right)\prod_{i=1}^\ell\left(1-e^{-h_{i,j} |K_{v_i}|}\right)\right]\\
& =  \frac{1-e^{-\delta_\ell}}{1-e^{-1}} \sum_{j=1}^\ell \P_{p_c,\mathbf{h_j}}(\sA \cap \sD).
\end{align*}
Applying \eqref{eq:magmultupper} we deduce that 
\[
\P(\sA \cap \sD \cap \{|K_{v_j}|\leq n_j\})
\leq \frac{1-e^{-\delta_\ell}}{1-e^{-1}} C_2^\ell \prod_{i=1}^\ell \sqrt{h_{i,j}} = \frac{1-e^{-\delta_\ell}}{(1-e^{-1})\sqrt{\delta_\ell}} C_2^\ell \prod_{i=1}^\ell \sqrt{h_i}
\]
and hence by \eqref{eq:magmultlower} that
\begin{multline*}
\P_{p_c,\mathbf{h}}(\sB \cap \sD) \geq \P(\sA \cap \sD) - \sum_{j=1}^\ell \P(\sA \cap \sD \cap \{|K_{v_j}|\leq n_j\})\\ \geq \left[c_1^\ell - \frac{1-e^{-\delta_\ell}}{(1-e^{-1})\sqrt{\delta_\ell}} \ell C_2^\ell \right] \prod_{i=1}^\ell \sqrt{h_i}
\geq \frac{c_1^\ell}{2} \prod_{i=1}^\ell \sqrt{h_i}.
\end{multline*}
It follows that
\[
\bP_{p_c}\left(K_{v_1},\ldots,K_{v_\ell} \text{ are disjoint and } |K_{v_i}| > n_i \text{ for every $1 \leq i \leq \ell$}\right) \geq \frac{c^\ell \delta_\ell^{\ell/2}}{2} \prod_{i=1}^\ell n_i^{-1/2}
\]
for every $v_1,\ldots,v_\ell$ with $d(v_i,v_j)\geq K(\ell)$ for every $1\leq i < j \leq \ell$ and every $n_1,\ldots,n_\ell \geq 1$.
\end{proof}

\begin{proof}[Proof of \cref{thm:multiarm}]
It suffices to prove the lower bound, since the upper bound is an immediate consequence of \eqref{exponent:theta} and the BK inequality. 
Fix $\ell \geq 1$, and let $C_1$ be the constant from \eqref{eq:MagFKG}. Let $0<p<1$ and let $h>0$.  
Let $G[p]$ be Bernoulli-$p$ bond percolation, let $\cG_1,\ldots,\cG_\ell$ be independent ghost fields of intensity $h>0$, and write $\P_{p,h}$ for the joint law of $G[p]$ and $\cG_1,\ldots,\cG_\ell$.
Fix $p_c<p_0<p_{2\to 2}$. As in the proof of \cref{thm:multibody}, there exists $K(\ell)$ such that if $u,v$ are vertices of $G$ with distance at least $K(\ell)$ then we have by \cref{lem:KN_inverseBK,lem:diagrammatic} that
\[
T_{p}^2(u,v) \leq T_{p}^3(u,v) \leq \frac{p_c^{2+C_1\ell}}{64M \ell^2}.
\]
for every $0 \leq p \leq p_0$, where $M$ is the maximum degree of $G$. Thus, if $v_1,\ldots,v_\ell$ satisfy $d(v_i,v_j) \geq K(\ell)$ for every $1 \leq i < j \leq \ell$, then
\begin{equation*}\P_{p,h}\left(K_{v_1},\ldots,K_{v_\ell} \text{ are disjoint and $v_i \leftrightarrow \cG_i$ for every $1\leq i \leq \ell$}\right)
\geq \frac{1}{2}\left[\inf_{v\in V} M_{p,h}(v)\right]^\ell
\end{equation*}
for every $0 \leq p < p_0$ and $h>0$. Taking the limit as $h \downarrow 0$, we obtain that
\begin{equation*}\bP_p(v_1,\ldots,v_\ell \text{ are all in distinct infinite clusters})
\geq \frac{1}{2}\left[\inf_{v\in V} \bP_p(|K_v|=\infty)\right]^\ell
\end{equation*}
for every $0 \leq p < p_0$ and $v_1,\ldots,v_\ell$ satisfying $d(v_i,v_j) \geq K(\ell)$ for every $1 \leq i < j \leq \ell$. Thus, we may deduce \eqref{eq:twoarm1} from \eqref{exponent:theta}.

We now deduce the lower bounds of \eqref{eq:twoarm2} and \eqref{eq:twoarm3} from the lower bound of \eqref{eq:twoarm1}. First, notice that 
\begin{multline*}
\bP_{p_c+\frac{1}{n}}(K_{v_1},\ldots,K_{v_\ell} \text{  are disjoint and $\rad_\mathrm{int}(K_{v_i}) \geq n$ for every $1\leq i \leq \ell$})
\\
 \geq \bP_{p_c+\frac{1}{n}}(v_1,\ldots,v_\ell \text{ are all in distinct infinite clusters})
 \succeq_\ell n^{-\ell}
\end{multline*}
for every $n\geq 1$ and every  $v_1,\ldots,v_\ell$ satisfying $d(v_i,v_j) \geq K(\ell)$ for every $1 \leq i < j \leq \ell$: the first inequality is trivial and the second follows from \eqref{eq:twoarm1}. Now consider coupling percolation at $p_c$ and $p_c+\frac{1}{n}$ in the standard monotone fashion. If the clusters of $v_1,\ldots,v_\ell$ are all distinct with intrinsic radius at least $n$ in percolation at $p_c+\frac{1}{n}$, then the conditional probability that this continues to hold in percolation at $p_c$ is at least $[p_c/(p_c+\frac{1}{n})]^{n \ell}$. To see this, take a length-$n$ intrinsic geodesic from $v_i$ in each $(p_c+\frac{1}{n})$-cluster, and observe that, for the property to no longer hold at $p_c$, at least one edge in at least one of these geodesics must change from open to closed when we move from $p_c+\frac{1}{n}$ to $p_c$. Thus, we deduce that
\begin{multline*}
\bP_{p_c}(K_{v_1},\ldots,K_{v_\ell} \text{  are disjoint and $\rad_\mathrm{int}(K_{v_i}) \geq n$ for every $1\leq i \leq \ell$})\\
\geq \left[\frac{p_c}{p_c+\frac{1}{n}}\right]^{n\ell} \bP_{p_c+\frac{1}{n}}(K_{v_1},\ldots,K_{v_\ell} \text{  are disjoint and $\rad_\mathrm{int}(K_{v_i}) \geq n$ for every $1\leq i \leq \ell$}),
\end{multline*}
so that the lower bound of \eqref{eq:twoarm2} follows from the lower bound of \eqref{eq:twoarm1}. Finally, the lower bound of \eqref{eq:twoarm3} can be deduced from the lower bound of \eqref{eq:twoarm2} using \cref{prop:intdecay} in a very similar manner to the proof of \cref{thm:ext_radius}.
\end{proof}

Let $G$ be a connected, locally finite graph and consider a Bernoulli bond percolation $G[p]$ on $G$. A vertex $v$ of $G$ is said to be a \textbf{furcation point} if $K_v$ is infinite and deleting $v$ would split $K_v$ into at least three distinct infinite clusters, and say that $v$ is a \textbf{trifurcation point} if deleting $v$ would split $K_v$ into \emph{exactly} three distinct infinite clusters. These points come up memorably in the Burton-Keane \cite{burton1989density} proof of uniqueness of the infinite cluster in amenable transitive graphs, where it is argued that if $G[p]$ has infinitely many infinite clusters, then it must have furcation points; see also \cite[Section 7.3]{LP:book}. Using \cref{thm:multiarm} allows one to make this proof quantitative, leading to the following corollary. (Note that the upper bound follows trivially from \eqref{exponent:theta} and the BK inequality.)

\medskip

\begin{corollary}
\label{cor:furcations}
Let $G$ be a connected, locally finite, quasi-transitive graph, and suppose that $p_c<p_{2\to 2}$. Then there exists a vertex $v$ of $G$ such that
	\[
\bP_p\bigl(v \text{ \emph{is a trifurcation point}}\bigr) \asymp (p-p_c)^3 \qquad  \text{ as $p \downarrow p_c$.}
	\]
\end{corollary}

(Note that, under the triangle condition, being a furcation point but \emph{not} a trifurcation point has probability of order at most $(p-p_c)^4$ as $p \downarrow p_c$ by \eqref{exponent:theta} and the BK inequality.)

\section{Applications to percolation in the hyperbolic plane}
\label{sec:planar}

In this section we apply the results of the previous sections to study percolation on quasi-transitive nonamenable \emph{simply connected planar maps} with locally finite dual. 
In particular, we apply duality arguments to compute various critical exponents governing the geometry of clusters at the \emph{uniqueness threshold} $p_u$. 

Let us now briefly recall the relevant terminology. A (locally finite) \textbf{map} is a   connected, locally finite graph $G$ together with an equivalence class of proper embeddings of $G$ into orientable surfaces (without boundary), where we consider two  embeddings to be equivalent if there is an orientation preserving homeomorphism between the two surfaces sending one embedding to the other. A map is said to be \textbf{simply connected} if the surface it is embedded in is (in which case it is homeomorphic to the plane or the sphere). Maps can also be defined combinatorially as graphs equipped  with cyclic orderings of the oriented edges emanating from each vertex. Every map $M$ has a \textbf{dual} $M^\dagger$ whose vertices are the faces of $M$ and whose faces correspond to the vertices of $M$, and there is a natural bijection between the edges of $M$ and the edges of $M^\dagger$ sending each edge $e$ of $M$ to the unique edge $e^\dagger$ of $M^\dagger$ that crosses $e$. See \cite{LZ} or \cite[Section 2.1]{unimodular2} for further background. Let us remark again that if $M$ is a simply connected, quasi-transitive, locally finite, nonamenable map then it is Gromov hyperbolic \cite{MR3658330} and hence has $p_c<p_{2\to 2}$ by the results of \cite{1804.10191}.

Let $M$ be a map with locally finite dual $M^\dagger$. Observe that if $\omega$ is distributed as Bernoulli-$p$ bond percolation on $M$, then the configuration $\omega^\dagger \in \{0,1\}^{E^\dagger}$  defined by $\omega^\dagger(e^\dagger)=1-\omega(e)$ is distributed as Bernoulli-$(1-p)$ percolation on $M^\dagger$. This duality is extremely useful in the case that $M$ is simply connected. In particular, 
Benjamini and Schramm proved that if $M$ is a nonamenable, quasi-transitive, locally finite, simply connected map with locally finite dual, then 
\begin{equation}
\label{eq:pcpu_duality}
 p_u(M)=1-p_c(M^\dagger) \qquad \text{ and } \qquad 0<p_c(M)<p_u(M)<1.
\end{equation} See also the earlier work of Lalley \cite{MR1614583}. 

We now apply this duality  to deduce various results about percolation at $p_u$ on $M$ by converting them into results about \emph{critical} percolation on $M^\dagger$.  
Given two vertices $x$ and $y$ that are in the same cluster, we write $\operatorname{ConRad}(x,y)$ for the minimal $r$ such that $x$ and $y$ are connected by an open path in the union of extrinsic balls $B(x,r)\cup B(y,r)$.

\begin{thm}
\label{thm:planarpuexponents}
Let $M$ be a connected, locally finite, nonamenable, quasi-transitive, simply connected planar map with locally finite dual. Then 
\begin{align*}
\sup_{x \sim y}\bP_{p_u}\bigl[ d_\mathrm{int}(x,y) &\geq n \mid x \leftrightarrow y \bigr] \asymp n^{-1} &&\text{and}\\
\sup_{x\sim y}\bP_{p_u}\bigl[ \operatorname{ConRad}(x,y) &\geq n \mid x \leftrightarrow y \bigr] \asymp n^{-2} &&
\end{align*}
where the suprema are taken over all neighbouring pairs of vertices in $M$.
\end{thm}

The reader may find it interesting to compare these exponents to the corresponding exponents for the free uniform spanning forest \cite{HutNach15b}, which are $1/2$ and $1$ respectively. \cref{thm:planarpuexponents} is complemented by subsequent work of the author with Hermon \cite{SupercriticalNonamenable}, which establishes in particular that, under the hypotheses of \cref{thm:planarpuexponents}, for every $p\in (0,p_u) \cup (p_u,1)$, there exists a constant $c_p>0$ such that
\begin{equation}
\bP_{p}\bigl[ d_\mathrm{int}(x,y) \geq n \mid x \leftrightarrow y \bigr] \leq e^{-c_p n}
\end{equation}
for every $n\geq 1$ and every pair of neighbouring vertices $x$ and $y$. Finally, we remark that \cref{thm:planarpuexponents} settles a problem raised by Gabor Pete \cite[Exercise*** 12.68]{Pete}.

\medskip

Note that the lower bounds of \cref{thm:planarpuexponents}  cannot be strengthened to apply to \emph{every} edge, since the local geometry of the graph might make it impossible for the endpoints of some edges to be in distinct large clusters. The proof shows that the lower bound holds whenever the edge dual to that between $x$ and $y$ lies on a doubly infinite geodesic in $M^\dagger$.

\begin{proof}[Proof of \cref{thm:planarpuexponents}]
Let $\omega$ be $p_u$ percolation on $M$. Then $\omega^\dagger$ is $p_c$ percolation on $M^\dagger$ and consequently has only finite clusters by the results of \cite{BLPS99b,Hutchcroft2016944} since $M^\dagger$ is quasi-transitive and nonamenable.
Let $e$ be an edge of $M$, let $f$ and $g$ be the two faces of $M$ on either side of $e$, and let $K_1$ and $K_2$ be the clusters of $f$ and $g$ in $\omega^\dagger \setminus \{e^\dagger\}$ respectively. 
Observe that if $x$ and $y$ are connected in $\omega$ with $d_\mathrm{int}(x,y)>1$ then we must have that $\omega(e)=0$ and that $K_1 \neq K_2$,  since otherwise the path between $f$ and $g$ in $\omega^\dagger \setminus \{e^\dagger\}$ would disconnect $x$ from $y$. Moreover, we claim that there exist positive constants $c,C$ such that
\begin{equation}
\label{eq:dint_duality}
c \min\{ |K_1|, |K_2| \} \leq d_\mathrm{int}(x,y) \leq C\min\{ |K_1|, |K_2| \}
\end{equation}
and
\begin{equation}
\label{eq:conrad_duality}
c \min\{ \diam(K_1), \diam(K_2) \} \leq \operatorname{ConRad}(x,y) \leq C\min\{ \diam(K_1), \diam(K_2) \}
\end{equation}
on this event. 
Indeed, if $\eta$ is any open path from $x$ to $y$ in $\omega \setminus \{e\}$ then $\eta \cup \{e\}$ forms a cycle that surrounds one of the two clusters $K_1$ or $K_2$. The lower bound of \eqref{eq:dint_duality} follows by nonamenability of $M$ together with the fact that $M$ and $M^\dagger$ both have bounded degrees, these facts together implying that any simple cycle in $M$ must have length comparable to the number of faces it bounds. For the upper bound of \eqref{eq:dint_duality}, we note that the collection of dual edges other than $e$ in the boundary of $K_i$ must contain a primal path from $x$ to $y$ on the event that $K_1$ and $K_2$ are distinct. Thus, $d_\mathrm{int}(x,y)$ is at most this number of dual edges, which is at most a constant multiple of $|K_i|$ since $M^\dagger$ has bounded degrees.
The bounds of \eqref{eq:conrad_duality} follow by similar arguments.

The upper bounds of the theorem follow immediately from the estimates \eqref{eq:dint_duality}, \eqref{eq:conrad_duality}, the bounds \eqref{exponent:volume} and of \cref{thm:ext_radius} (applied to $M^\dagger$) together with the BK inequality. On the other hand, the lower bounds are very close to those of \cref{thm:multibody,thm:multiarm} and can be deduced from those theorems via a finite energy argument. We will give only a brief sketch of how this is done; an essentially identical argument is given in detail in the proof of \cite[Theorem 1.3]{hutchcroft2018locality}. By \cref{thm:multibody,thm:multiarm} applied to $M^\dagger$, there exist positive constants $c$ and $k$ such that if $f,g$ are faces of $M$ with distance at least $k$ (in $M^\dagger$) then 
\begin{align*}
\bP(K_{f},K_{g} \text{ distinct and } |K_{f}|,|K_{g}| \geq n) &\geq \frac{c}{n} \intertext{ and } \bP(K_{f},K_{g} \text{ distinct and } \diam(K_{f}),\diam(K_{g}) \geq n) &\geq  \frac{c}{n^2}.
\end{align*}
Fix two such $f$ and $g$. If we start with a configuration in which $K_f$ and $K_g$ are distinct and then force the edges along the geodesic in $M^\dagger$ between $f$ and $g$ to be open one at a time, there must be some first time when the two clusters merge when the next edge is added. At the time immediately before this one, the edge about to be added has endpoints in distinct clusters, both of which are at least as large (both in terms of volume and diameter) as $K_f$ and $K_g$ were in the original configuration. Since there are a constant number of choices of what this edge could be, and since forcing a bounded number of edges to be included in the percolation configuration increases probabilities of events by at most a constant, the claim follows.
\end{proof}


\begin{remark}
Outside of the planar case, the behaviour of percolation at $p_u$ is very poorly understood. In particular, no good characterisation of whether or not there is uniqueness at $p_u$ is known. It is known that quasi-transitive simply connected planar maps have uniqueness at $p_u$, but that graphs defined as infinite products \cite{MR1770624} and Cayley graphs of infinite Kazhdan groups  \cite{LS99} have nonuniqueness at $p_u$. It is open whether uniform lattices in $\bbH^d$ with $d\geq 3$ have nonuniqueness at $p_u$ or not. Indeed, this is also open for transitive graphs that are rough-isometric to $\bbH^2$ but are not planar. In \cite[Section 6]{1804.10191} it is shown that there is always nonuniqueness at $p_{2\to 2}$, and one possibility is that there is uniqueness at $p_u$ if and only if $p_u > p_{2\to 2}$.
\end{remark}

\subsection*{Acknowledgments} We thank Gady Kozma for helpful discussions on interpolation theory of operators, and thank Asaf Nachmias for comments on a draft.
We also thank the anonymous referees for their careful reading and helpful comments.

  \setstretch{1}
  \bibliographystyle{abbrv}
  \bibliography{unimodularthesis.bib}
\end{document}